\documentclass[12pt, a4paper]{amsart}
\usepackage[margin=1.7cm]{geometry}
\usepackage{etex}
\usepackage{enumerate}
\usepackage[english]{babel}
\usepackage{comment}
\usepackage[T1]{fontenc}
\usepackage{amscd}
\usepackage{color}
\usepackage{amssymb}
\usepackage{amsmath}
\usepackage{amsfonts}
\usepackage{amssymb}
\usepackage{amssymb}
\usepackage{amsthm}

\theoremstyle{definition}
\newtheorem{definition}{Definition}[section]
\newtheorem{example}[definition]{Example}
\newtheorem{remark}[definition]{Remark}

\theoremstyle{plain}
\newtheorem{theorem}[definition]{Theorem}
\newtheorem*{theorem*}{Theorem}
\newtheorem{proposition}[definition]{Proposition}
\newtheorem{lemma}[definition]{Lemma}
\newtheorem{corollary}[definition]{Corollary}
\numberwithin{equation}{section}
\def \alt96 {`}
\def \RN {\mathbb{R}^N}
\def \R {\mathbb{R}}

\def \loc {\mathrm{loc}}
\def \N {\mathbb{N}}

\def \e {\varepsilon }
\def \LL {{\mathcal{L}}}
\def \II {\mathcal{I}}
\def \d {\mathrm{d}}
\def \de {\partial}

\def \longto {\longrightarrow}

\def \OO {\mathcal{O}}
\def \CCs {\mathcal{C}_s}
\def \spX {\mathbb{X}(\Omega)}

\begin{document}
 \title[Mixed operators: regularity and maximum principles]
 {Mixed local and nonlocal elliptic operators: \\ regularity and maximum principles}
% ---------------------------------------- Authors ----------------
 \author[S.\,Biagi]{Stefano Biagi}
 \author[S.\,Dipierro]{Serena Dipierro}
 \author[E.\,Valdinoci]{Enrico Valdinoci}
 \author[E.\,Vecchi]{Eugenio Vecchi}
 
 \address[S.\,Biagi]{Dipartimento di Matematica
 \newline\indent Politecnico di Milano \newline\indent
 Via Bonardi 9, 20133 Milano, Italy}
 \email{stefano.biagi@polimi.it}
 
 \address[S.\,Dipierro]{Department of Mathematics and Statistics
 \newline\indent University of Western Australia \newline\indent
 35 Stirling Highway, WA 6009 Crawley, Australia}
 \email{serena.dipierro@uwa.edu.au}
 
 \address[E.\,Valdinoci]{Department of Mathematics and Statistics
 \newline\indent University of Western Australia \newline\indent
 35 Stirling Highway, WA 6009 Crawley, Australia}
 \email{enrico.valdinoci@uwa.edu.au}
 
 \address[E.\,Vecchi]{Dipartimento di Matematica
 \newline\indent Università di Bologna \newline\indent
 Piazza di Porta San Donato 5, 40126 Bologna, Italy}
 \email{eugenio.vecchi@polimi.it}

\subjclass[2010]{35A01, 35B65, 35R11}

\keywords{Operators of mixed order, existence, regularity,
maximum principle, qualitative properties of solutions}

\thanks{The authors are members of INdAM. S. Biagi
is partially supported by the INdAM-GNAMPA project 
\emph{Metodi topologici per problemi al contorno associati a certe 
classi di equazioni alle derivate parziali}.
S. Dipierro and E. Valdinoci are members of AustMS.
S. Dipierro is supported by
the Australian Research Council DECRA DE180100957
``PDEs, free boundaries and applications''. 
E. Valdinoci is supported by the Australian Laureate Fellowship
FL190100081 ``Minimal surfaces, free boundaries and partial differential equations''. 
E. Vecchi is partially supported
by the INdAM-GNAMPA project 
``Convergenze variazionali per funzionali
e operatori dipendenti da campi vettoriali''. We thank the Referees for their very valuable comments.}

\date{\today}
% --------------------------------------------------------------------------
 \begin{abstract}
We develop a systematic study of the superpositions
of elliptic operators with different orders, mixing classical and fractional
scenarios. For concreteness, we focus on the sum of the Laplacian and the fractional Laplacian,
and we provide structural results, including existence,
maximum principles (both for weak and classical solutions),
interior Sobolev regularity
and boun\-da\-ry regularity of Lipschitz type.
 \end{abstract}
 
\maketitle
 
\section{Introduction} \label{sec.intro}

The goal of this paper is to develop a systematic study
of mixed operators. The word ``mixed'' refers here to
the differential (or pseudo-differential) order of the operator,
and to the type of the operator, which combines classical
and fractional features.

Though many of the techniques that we present here
are rather ``general'', for the sake of concreteness,
and not to hide the main flow of ideas by technical complications,
we focus here on an operator which deals with the coexistence
of a Laplacian and a fractional Laplacian, given by
\begin{equation}\label{LOPERATO1}
\LL : = -\Delta + (-\Delta)^s,\qquad{\mbox{ for some }}s\in(0,1).
 \end{equation}
Here, $(-\Delta)^s$ is the nonlocal operator defined as
 \begin{equation} \label{eq:defdeltas}
  (-\Delta)^s u(x) = c_{N,s}\; \textrm{P.V.}\int_{\RN}\frac{u(x)-u(y)}{|x-y|^{N+2s}}\,\d y,
  \end{equation}
 where $c_{N,s}$ is a suitable normalizing constant, whose explicit expression is given by
   $$c_{N,s} = 
   \bigg(\int_{\RN}\frac{1-\cos(\zeta_1)}{|\zeta|^{N+2s}}\,\d\zeta\bigg)^{-1},$$
   and, as customary, ``$\textrm{P.V.}$'' stands for the Cauchy's principal value,
   see e.g. formulas~(3.1) and~(3.2) in~\cite{guida}.
   
We recall that the literature presents
 several
 variants of the fractional Laplacian, including one that is defined in terms of the eigenbasis and spectrum of the
Dirichlet Laplacian on a bounded domain and one
in which the singular integral only extends to the given domain.
These are quite different operators from the one in~\eqref{eq:defdeltas}, see e.g.~\cite[Sections~2.1, 2.2. 2.3, 4.2 and~4.3]{MR3967804} for a list of similarities and differences between these fractional operators.
   
Operators as in~\eqref{eq:defdeltas} arise naturally
from the superposition of two stochastic processes
with different scales (namely, a classical random
walk and a L\'evy flight): roughly speaking, when a particle can
follow either of these two processes according to a certain probability,
the associated limit diffusion equation is described by
an operator of the form described in~\eqref{LOPERATO1}:
see in particular the appendix in~\cite{PRO}
for a
thorough discussion of this phenomenon and~\cite{PRECISO}
for the description of a biological population in an ecological niche
modeled by a mixed operator.

In view of this motivation, we think that operators as
in~\eqref{LOPERATO1} will enjoy a constantly rising popularity
in applied sciences, also to study the different impact
of ``local'' and ``nonlocal'' diffusions in concrete situations
(e.g. how different types of ``regional'' or ``global''
restrictions may reduce the spreading of a pandemic disease, see e.g.~\cite{TRAV}).
Other classical applications include
heat transport in magnetized
plasmas (see~\cite{PLA}).
\medskip

The mathematical study of operators with different order
is not new in itself, and indeed the literature
already presents results concerning, among the others,
the theory of viscosity solutions
(see~\cite{JAK1, JAK2, MR2422079, OHKSBDc3847t8567, CIOM, BBGG-1, BBGG-2}),
the Aubry-Mather theory for sums of different
fractional Laplacians
(see Remark~5.6 in~\cite{LL}), 
regularization effects of Cahn-Hilliard equations
(see~\cite{MAG}), numerics (\cite{NUMERI}),
probability
and stochastics (see~\cite{CHEN1, CHEN2, PROS}),
symmetry results
for mixed range phase transitions (see~\cite{CASE}),
porous medium equations
(see~\cite{TESO}),
decay estimates for parabolic equations (see~\cite{VES}),
specific Liouville theorems for systems
of equations driven by sums of fractional Laplacians (see~\cite{JJ, ALIBAUD}),
fractional damping effects (see~\cite{PATA}),
and Bernstein-type regularity results (see~\cite{BERN}).

Though conceptually different,
the problems related to mixed order equations are closely
related in spirit to the ones of variable order equations
(see~\cite{HO, BK1, BK, LIKE}), which have themselves concrete
significance in applied sciences (see e.g.~\cite{WW, CUORE}).
\medskip

The main focus here
is on the operator in~\eqref{LOPERATO1} with the aim
of obtaining a number of {\em structural results} (based on techniques
which we plan to extend to more general situations in future works).
These results deal with
distributional as well as classical solutions, and they
can be grouped into four categories: existence,
maximum principles,
interior regularity,
and boundary regularity.\medskip

Let us now describe the main results in further detail.
First of all, we will introduce a suitable notion of weak solution
associated to the operator in~\eqref{LOPERATO1}.
In light of the mixed nature of the problem, this step already presents
some caveats, due to the possible choices of including or excluding
the external data within the classical Sobolev spaces framework.
Our setting for weak solutions will be described in
Definition~\ref{def.weaksol}, and then exploited
to obtain,
via the Lax-Milgram's theorem,
the following existence result, of very
classical flavor. In all the forthcoming statements,
 we tacitly understand that $\Omega\subseteq\RN$ is a \emph{bounded open set with
 $C^1$ boundary}.
 
\begin{theorem} \label{thm.existenceLax}
	Given $f\in L^2(\Omega)$, there exists a unique
     \emph{weak solution}
     $u_f\in H^1(\RN)$ of $$
     \begin{cases}
    \LL u = f & \text{in $\Omega$}, \\
    u\big|_{\RN\setminus \Omega} = 0.
   \end{cases}$$
     Furthermore, there exists a $\mathbf{c} > 0$ such that
     $$\|u_f\|_{H^1(\RN)}\leq\mathbf{c}\,\|f\|_{L^2(\Omega)}.$$
\end{theorem} 
We then focus on the maximum principles
associated to the operator in~\eqref{LOPERATO1}.
Their formulation is slightly different for weak
and classical solutions. More precisely, we
present a weak maximum principle for weak solutions, as follows:

  \begin{theorem}\label{thm.WMPweak}
Let $u\in H^1(\RN)$ weakly satisfy~$\LL u\ge0$ in~$\Omega$.
If $u\geq 0$ a.e.\,in $\RN\setminus\Omega$, then
$u\geq 0$ a.e.\,on $\Omega$.  \end{theorem}
    
For classical solutions, this statement is strengthened in the following result:
    
  \begin{theorem}\label{NUOVO}
  Let $u\in C(\RN,\R)\cap C^2(\Omega,\R)$, with
  $$\int_{\RN}\frac{|u(x)|}{1+|x|^{N+2s}}\,\d x
    <\infty
  .$$ Suppose that
  $$\begin{cases}
   \LL u\geq 0 & \text{pointwise in $\Omega$}, \\
   u \geq 0 & \text{in $\RN\setminus\Omega$}.
  \end{cases}$$
  Then
  \begin{equation}\label{thm.WMPLL}
  u\geq 0\quad{\mbox{ in }}\quad\Omega.\end{equation}

Furthermore,
  \begin{equation}\label{thm.SMPLL}
  {\mbox{if there exists $x_0\in\Omega$ such that
  $u(x_0) = 0$, then 
$u\equiv 0$ throughout $\RN$.}}\end{equation}
 \end{theorem}

As customary, the statement in~\eqref{thm.WMPLL} can be considered as a weak maximum principle
and the one in~\eqref{thm.SMPLL} as a strong maximum principle.
The difference between the weak maximum principle in Theorem~\ref{thm.WMPweak}
and that in~\eqref{thm.WMPLL} is in the assumptions required on~$u$
(the Sobolev setting being considered in Theorem~\ref{thm.WMPweak} and the classical one in Theorem~\ref{NUOVO}).

Though these maximum principles are of classical
flavor, we remark that their validity depends sensibly on the
type of the operator and on the setting of the data.
In particular, we will show in
Appendix~\ref{sec.appendix} that the maximum principle does not hold
if the external condition ``$u\geq 0$ 
in $\RN\setminus\Omega$''
is replaced by ``$u\ge0$
  on $\de\Omega$'': namely, classical boundary conditions
  are not
   enough to guarantee the validity of maximum principles for the operator
  in~\eqref{LOPERATO1}, notwithstanding the presence
  of the classical Laplacian in this operator.
  
Furthermore, we will show that these
maximum principles do not hold when one of the elliptic operators
in~\eqref{LOPERATO1} has the ``wrong sign'', e.g. for operators
of the type~$\Delta + (-\Delta)^s$.

The fact that the weak maximum principles
in Theorem \ref{thm.WMPweak} and in~\eqref{thm.WMPLL} of Theorem~\ref{NUOVO}
hold
for $\LL$ but not for similar operators
with the ``wrong sign'' is very reasonable, especially
in view of some potential-theoretic results of
the early '60s (see, e.g., \cite{BCP, Courrege}).
In fact, in the aforementioned papers is proved that: if 
$$A:C_0^2(\RN)\to C(\RN)$$
is a linear operator, 
then the next two conditions are \emph{equivalent}:
\begin{itemize}
 \item[(i)] $A$ is continuous (with respect to suitable topologies) and
 $$
 \begin{cases}
 u\in C_0^2(\RN), \\
 \text{$u(x) = \inf_{\RN}u \leq 0$}
 \end{cases}\,\,\Longrightarrow\,\, Au(x)\leq 0;
 $$
 \item[(ii)] for every $u\in C^2_0(\RN)$, we have
 \begin{equation} \label{eq.formAPS}
 \begin{split}
  Au & = -\sum_{i,j = 1}^Na_{i,j}(x)\frac{\de^2 u}{\de x_i\de x_j}(x)
  + \sum_{i = 1}^Nb_i(x)\frac{\de u}{\de x_i}(x)+c(x)u(x) \\[0.1cm]
  & \qquad + \int_{\RN}s(x,\d y)\big[u(x)-u(y)-\mathbf{1}_{\{|x-y|\leq 1\}}\cdot
  \langle \nabla u(x),x-y\rangle \big],
  \end{split}
  \end{equation}  
  where $A(x) = \begin{pmatrix}
  a_{i,j}(x)
  \end{pmatrix}$
  is \emph{positive definite} for every $x\in\RN$, 
  $c\leq 0$ on the whole of $\RN$ and 
  $s$ is a measurable kernel on $\RN$ satisfying the properties
  \begin{itemize}
   \item[$\bullet$] $s(x,\{x\}) = 0$;
   \item[$\bullet$] for every positive function $f\in C_0(\RN)$, the map
   $$x\mapsto \int_{\RN}s(x,\d y)|y-x|^2f(y)\,\d y\quad\text{is measurable}.$$
  \end{itemize}
 \end{itemize}
 In particular, property (i) (which shall be used in the proof of~\eqref{thm.WMPLL})
 holds for 
 $\LL$ 
 but not for
 similar operators with ``wrong sign''. It is worth mentioning that
 in the papers \cite{BCP, Courrege} it is not studied the validity
 of a weak maximum principle as in~\eqref{thm.SMPLL} (i.e.,
 the possibility of ``propagating'' the sign of $u$ from $\RN\setminus\Omega$
 into $\Omega$). On the other hand, the equivalence of (i) and (ii) is exploited
 in the recent paper \cite{ALIBAUD} to characterize all the operators
 of the form \eqref{eq.formAPS} for which a Liouville-type theorem holds.
\medskip

We devote part of this paper to the interior and boundary
regularity properties of solutions.
Though a variety of different directions
can be taken for this, we focus here on the interior regularity
theory in Sobolev spaces. The main result in this sense
goes as follows:

\begin{theorem} \label{thm.higherregul}
      Let $m\in\mathbb{N}\cup\{0\}$ and
      $f\in H^m(\Omega)$. Assume that~$u\in H^1(\RN)$ is a weak solution
      of $$      \LL u = f\qquad\text{in $\Omega$}.$$
      Then
      $u\in H^{m+2}_{\loc}(\Omega)$.
    \end{theorem}   
The proof of Theorem~\ref{thm.higherregul} requires some technical
improvements with respect to the
classical energy methods
and bootstrap arguments,
since the fractional operator prevents
the possibility of taking derivatives of the equation.
To overcome this difficulty, we will
exploit truncation arguments and difference quotients
in a suitable way.\medskip

As for the boundary regularity, for concreteness
we focus on the linear growth
and Lipschitz regularity for solutions in convex domain.
Our result can be summarized as follows:
\begin{theorem}\label{BDTH}
Assume that $\Omega$ is strictly convex
and let $\bar C>0$. Let 
    $u\in H^1(\RN)$ be such that
    \begin{equation}\label{eq.PbGeneralWeak}
     \begin{cases}
	  \LL u\le \bar C & \text{in $\Omega$},\\
	   u=0 &  \text{in $\RN\setminus\Omega$},\\
	   u\le\bar C & \text{in $\RN$}.
     \end{cases}
    \end{equation}
    Then, there exist $C$, $\ell>0$ such that,
    for every $p\in\partial\Omega$, we have that
    \begin{equation} \label{Y:1}
	u(x)\le C\,\bar C\, |x-p| \qquad\text{for a.e.\,$x\in B(p,\ell)$.}
	\end{equation}
   \end{theorem}

As customary, the notation~$B(p,\ell)$
here denotes the Euclidean ball centered at~$p$ with radius~$\ell$.

Interestingly, the boundary regularity in Theorem~\ref{BDTH}
is stronger than in the case of the fractional Laplacian, in which the
solution is in general not better than~$C^s$, see~\cite{RosSerra}.
\medskip

   As a byproduct of Theorem~\ref{BDTH}, one obtains
   also existence and regularity results, as given
  by the following result: 
    \begin{theorem}\label{thm.RegulBoundaryII}
     Assume that $\Omega$ is strictly convex. 
     Let~$m \geq \frac{N}2+3$ and
     $f\in C^m(\Omega,\R)\cap L^\infty(\Omega)$.
     Then, there exists a unique classical solution
     $\mathfrak{u}_f\in C(\RN,\R)\cap C^2(\Omega,\R)$ 
     of
     \begin{equation} \label{eq.BVPIntro}
      \begin{cases}
    \LL \mathfrak{u}_f = f & \text{in $\Omega$}, \\
    \mathfrak{u}_f\big|_{\RN\setminus \Omega} = 0,
   \end{cases}
	\end{equation}
   with
  $$\int_{\RN}\frac{|\mathfrak{u}_f(x)|}{1+|x|^{N+2s}}\,\d x <\infty .$$  
Moreover, this solution $\mathfrak{u}_f$ satisfies the following
additional properties: 
     \begin{itemize}
	 \item[(i)] $\mathfrak{u}_f\in H^1(\RN)$;
	 \item[(ii)] $\|\mathfrak{u}_f\|_{L^\infty(\RN)}\leq \mathbf{c}\,\|f\|_{L^\infty(\Omega)}$; 
	 \item[(iii)] for every $p\in\de\Omega$ there exists $\ell > 0$ such that
	 $$|\mathfrak{u}_f(x)|\leq \mathbf{c}\,\|f\|_{L^\infty(\Omega)}\cdot
	 |x-p|\qquad\text{for all $x\in\Omega\cap B(p,\ell)$}.$$    
     \end{itemize}   
     Here, $\mathbf{c} > 0$ is a constant independent of $\mathfrak{u}_f$.       
    \end{theorem}
    
Let us comment a bit about the ``philosophy'' of the regularity theory developed
 in this paper.  At first, in view of the classical regularity theory for the Laplacian,
 one could believe that adding an
extra fractional Laplacian to an already elliptic problem would just produce ``expected''
outcomes, as a lower order effect. For instance, one could argue that
interior regularity results (such as Theorem~\ref{thm.higherregul}, as well as the forthcoming Theorem~\ref{thm.mainregul}) could be a consequence of the classical theory.
Specifically: on the one hand, it is known that, if~$g\in L^2(\Omega)$ and~$u\in H^1(\Omega)$ are such that
\begin{equation}\label{021ur90327ytr3gtgPJS93uf9043yhg98u4} \begin{cases}
-\Delta u = g & {\mbox{ in }}\Omega,\\
u = 0 &{\mbox{ on }}\partial\Omega,\end{cases}\end{equation}
then $u\in H^2_{\loc}(\Omega)$ and the corresponding norm is controlled, up to constants, by~$\|u\|_{L^2(\Omega)}+
\|g\|_{L^2(\Omega)}$.

On the other hand, if we have a weak solution~$u$ to the
mixed problem, one can set~$ g := f -(-\Delta)^su$ and we reduce the situation
to the classical problem as
formulated in~\eqref{021ur90327ytr3gtgPJS93uf9043yhg98u4}.

These observations yield the regularity in class~$H^2_{\loc}(\Omega)$ with norm controlled
by~$\|u\|_{L^2(\Omega)}+
\|(-\Delta)^s u\|_{L^2(\Omega)}+\|f\|_{L^2(\Omega)}$,
leading to even better statements than the ones provided here
(e.g. in the forthcoming estimate~\eqref{eq.estimregulH2}), {\em but only when~$s$ is below~$1/2$}
(and above this threshold the norm in~$H^s(\Omega)$ would be too strong compared
with the initial regularity of the solution in~$H^1(\Omega)$).

That is: for the low-range of~$s$, the nonlocal part of the operator can be treated as a perturbation of the classical elliptic problem but for the high-range of fractional exponents these tricks seem to be unavailable
and the operator must be treated on its own terms (roughly speaking,
one can efficiently consider the fractional Laplacian as a lower order perturbation only when
the fractional exponent is ``sufficiently small'').

This is perhaps not a merely technical aspect of the problem: for instance, a distinction of this
sort will pop up also in the construction of the barrier constructed to prove Theorem~\ref{BDTH},
since for~$s$ below~$1/2$ one would not need the iteration exploited to reabsorb the nonlocal contributions
    and the proof would significantly simplify.

Similar occurrence of different phenomena according to the fractional threshold~$1/2$
appear in the literature in several descriptions of nonlocal problems, see e.g.~\cite{CAFSOU, SVMA, DISLO, SER}
and the references therein. 

For all these reasons, though in principle other approaches could be possible (relying e.g. on intermediate estimates to bootstrap, scaled norms and interpolation theory), we think it is useful and instructive
to develop a series of analytic tools which account for mixed operators in their whole complexity
rather than limiting our vision to perturbation methods from the classical cases. 
    \medskip
    
After this article was completed and posted online,
  we have re\-cei\-ved the very interesting pre\-print \cite{NIC-MAT},
  which considered a nonlinear problem of mixed type.
  The motivations, methodologies and results obtained are different
  from the ones in this paper, but Theorem~1.1 in~\cite{NIC-MAT}
  is related to Theorem~\ref{thm.RegulBoundaryII} here.
  Complementary to this result,
  we stress that Theorem 1.4 in \cite{NIC-MAT} shows that 
  there exists no classical solutions of \eqref{eq.BVPIntro} when $f$
  is merely in $L^\infty_{\loc}(\Omega)$.
  
  Finally, we mention that similar results could be obtained via the 
  probabilistic methods and the Green representation formulas dealt with 
  in \cite{CHEN1, CHEN2}. For regularity results involving mixed order diffusive operators
  in different directions see also~\cite{BBGG-1, FARINA1, FARINA2}.
\medskip

The proof of the boundary regularity result in Theorem~\ref{BDTH} relies
on the introduction of an explicit barrier. To the best of our
knowledge, this barrier
is completely new in the literature, and its construction
is based on an iterative method
of introducing ``correctors'' to recursively
compensate the terms
produced by the action of the nonlocal operator.
We think that this technique of iteratively canceling ``the nonlocal
tails'' is interesting in itself and can produce other results
in greater generality.\medskip

Some of the methodologies and motivations
presented in this paper will be also exploited in~\cite{NOI2}
to analyze the qualitative properties of solutions in specific
problems.
\medskip

The rest of this article is organized as follows.
In Section~\ref{sec.existence} we 
introduce the functional framework in which we work, discussing in particular
the notion
of weak solutions and giving the existence result in Theorem~\ref{thm.existenceLax}.
Section~\ref{sec.MPsLL} is devoted to the maximum principles, and to the proofs
of Theorems~\ref{thm.WMPweak} and~\ref{NUOVO}, and
Section~\ref{sec.regularity}
to the regularity theory, and to the proofs of Theorems~\ref{thm.higherregul}, \ref{BDTH}
and~\ref{thm.RegulBoundaryII}.

Then, in Appendix~\ref{sec.appendix}
we collect some counterexamples to the maximum principle.
 
  \section{Existence of weak solutions} \label{sec.existence}
   In this section we establish some basic facts on the existence
  of weak solutions for the Dirichlet problem associated with $\LL$, as defined in~\eqref{LOPERATO1},
  that is,
  \begin{equation} \label{eq.mainPB}
    \begin{cases}
    \LL u = f & \text{in $\Omega$}, \\
    u\big|_{\RN\setminus \Omega} = 0.
   \end{cases}
   \end{equation}
   Throughout the sequel, we tacitly understand that \emph{$\Omega\subseteq
   \RN$ is a bounded open set with $C^1$ boundary}. 
   Moreover, $s\in (0,1)$ is a fixed parameter,
   and~$(-\Delta)^s$ is as in \eqref{eq:defdeltas}.
 %  u(x) = c_{N,s}\; \textrm{P.V.}\int_{\RN}\frac{u(x)-u(y)}{|x-y|^{N+2s}}\,\d y$$
 %  where $c_{N,s}$ is a suitable normalizing constant, whose explicit expression is given by
 %  $$c_{N,s} = 
 %  \bigg(\int_{\RN}\frac{1-\cos(\zeta_1)}{|\zeta|^{N+2s}}\,\d\zeta\bigg)^{-1},$$
 %  and, as customary, ``$\textrm{P.V.}$'' stands for the Cauchy's principal value,
 %  see e.g. formulas~(3.1) and~(3.2) in~\cite{guida}.
   %
   In studying the sol\-va\-bi\-li\-ty of~\eqref{eq.mainPB},
   a `natural' space to consider is the following
   \begin{equation} \label{eq.defXOmega}
    \mathbb{X}(\Omega) := \big\{u\in H^1(\R^N):\,\text{$u\equiv 0$ in $\RN\setminus\Omega$}\big\}.
   \end{equation}
   We observe that, in view of the regularity assumption on $\Omega$,
   the space $\mathbb{X}(\Omega)$ is contained in~$ H^1(\RN)$ and
	is isomorphic
   to $H_0^1(\Omega)$ via the `zero-extension' map defined as
   $$\mathcal{E}_0:H_0^1(\RN)\to\mathbb{X}(\Omega), \qquad{\mbox{such that }}\quad
   \mathcal{E}_0(u) := u\cdot\chi_{\Omega}.$$
   As a consequence, $\mathbb{X}(\Omega)$ is endowed
   with a structure of (real) Hilbert space by
   the scalar product defined as follows
   $$\langle u, v\rangle_{\mathbb{X}} := \int_{\Omega}\langle \nabla u,\nabla v\rangle\,\d x
   \qquad \text{for all~$u$, $v\in\mathbb{X}(\Omega)$}.$$
   The norm associated with $\langle\cdot,\cdot\rangle_{\mathbb{X}}$ is
   $$\|u\|_{\mathbb{X}} := \|\nabla u\|_{L^2(\Omega)} \qquad
   \text{for all~$u\in\mathbb{X}(\Omega)$}, $$
   and $C_0^\infty(\Omega,\R)$ is \emph{dense} in $\mathbb{X}(\Omega)$.
   
   Furthermore,
   the classical Poincar\'{e} inequality holds in $\mathbb{X}$:
   more precisely, there exists a constant $\mathbf{c} > 0$ such that
   \begin{equation} \label{eq.PoincareX}
    \|u\|_{H^1(\RN)} \leq \mathbf{c}\,\|u\|_{\mathbb{X}}\qquad
   \text{for all $u\in \mathbb{X}(\Omega)$}.
   \end{equation}
   \medskip
   
   \noindent After all these preliminaries, 
   we can give the following definition.
   \begin{definition} \label{def.weaksol}
    Let $f\in L^2(\Omega)$.
    We say that a function $u\in H^1(\RN)$ is a \emph{weak solution
    of the equation}
    \begin{equation} \label{eq.mainPDE}
      \LL u = f\qquad\text{in $\Omega$}
      \end{equation}
      if, for every test function $\varphi\in C_0^\infty(\Omega,\R)$, one has
      \begin{equation} \label{eq.weaksoldef}
    \begin{split}
     \int_{\Omega}
     \langle\nabla u,\nabla \varphi\rangle\,\d x
      +\frac{c_{N,s}}{2}\int_{\RN\times\RN}
     \!\!\!\!\!\frac{(u(x)-u(y))(\varphi(x)-\varphi(y))}{|x-y|^{N+2s}}\,\d x\,\d y
   = \int_{\Omega}f\varphi\,\d x.
     \end{split}
    \end{equation}   
    Furthermore, one can more generally say that
	a function $u\in H^1(\RN)$ weakly satisfies~$\LL u\ge f$
	in~$\Omega$ if,
	for every nonnegative~$v\in \spX$, one has
     \begin{equation}\label{QUESTA}
    \begin{split}
     \int_{\Omega}
     \langle\nabla u,\nabla v\rangle\,\d x
     & +\frac{c_{N,s}}{2}\int_{\RN\times\RN}
     \!\!\!\!\!\frac{(u(x)-u(y))(v(x)-v(y))}{|x-y|^{N+2s}}\,\d x\,\d y
     \ge \int_{\Omega}fv\,\d x.
     \end{split}
    \end{equation}
	One can say that~$u$ weakly satisfies~$\LL u\le f$ in~$\Omega$
	if~$v:=-u$ weakly satisfies~$\LL v\ge f$
	in~$\Omega$.

    Finally, if $u$ is a weak solution of~\eqref{eq.mainPDE} and
    $u\in\mathbb{X}(\Omega)$, we say that 
    $u$ is a \emph{weak solution
    of problem}~\eqref{eq.mainPB}.
    \end{definition}

   \begin{remark} \label{rem.WeakSoluf}
     We notice that, if $u\in H^1(\RN)$ is any weak solution 
     of~\eqref{eq.mainPDE} (for some $f\in L^2(\Omega)$), 
     from the density of~$C_0^\infty(\Omega,\R)$ in $\spX$ we get
     \begin{equation} \label{eq.weaksoldefDensity}
    \begin{split}
     \int_{\Omega}
     \langle\nabla u,\nabla v\rangle\,\d x
     & +\frac{c_{N,s}}{2}\int_{\RN\times\RN}
     \frac{(u(x)-u(y))(v(x)-v(y))}{|x-y|^{N+2s}}\,\d x\,\d y
      \\
     & \quad= \int_{\Omega}fv\,\d x 
     \qquad\text{for all $v\in \spX$}.
     \end{split}
    \end{equation}
     In particular, if $u_f$ is a solution of~\eqref{eq.mainPB}
     (so that $u_f\in\spX$), we are entitled to use
      $u_f$ as a test function in~\eqref{eq.weaksoldefDensity},
     obtaining
     $$\int_{\Omega}|\nabla u_f|^2\,\d x
     + \frac{c_{N,s}}{2}\int_{\RN\times\RN}
     \frac{(u_f(x)-u_f(y))^2}{|x-y|^{N+2s}}\,\d x\,\d y
     = \int_{\Omega}fu_f\,\d x.$$
\end{remark}
 Having introduced the functional framework in which we work,
 we now prove the basic existence result in Theorem~\ref{thm.existenceLax}.
\begin{proof}[Proof of Theorem~\ref{thm.existenceLax}]
	We consider on the space $\mathbb{X}(\Omega)$ the bilinear form
     $B$ and the linear map $F$ defined, respectively, as follows:
     \begin{align*}
      & (\star)\,\,B(u,v) := \int_{\Omega}
     \langle\nabla u,\nabla v\rangle\,\d x+
     \frac{c_{N,s}}{2}\int_{\RN\times\RN}
     \!\!\!\!\!\!\!\frac{(u(x)-u(y))(v(x)-v(y))}{|x-y|^{N+2s}}\,\d x\,\d y, \\[0.2cm]
     & (\ast)\,\,F(u) := \int_{\Omega}fu\,\d x.
     \end{align*}
Using~\eqref{eq.PoincareX} and the fact that $H^1(\RN)$ is continuously
     embedded into $H^s(\RN)$, we see that both $B$ and $F$
     are (well-posed and) continuous with respect
     to the Hilbert structure of $\mathbb{X}(\Omega)$.
	 Moreover, we have that
	 \begin{align*}
	  B(u,u) & = \|\nabla u\|^2_{L^2(\Omega)}
	 + \frac{c_{N,s}}{2}\int_{\RN\times\RN}\frac{|u(x)-u(y)|^2}{|x-y|^{N+2s}}\,\d x\,\d y\ge
	 \|\nabla u\|^2_{L^2(\Omega)} 
	 = \|u\|^2_{\mathbb{X}(\Omega)},
     \end{align*}
      for all $u\in\mathbb{X}(\Omega)$.
     Hence, we are in the position to apply Lax-Milgram's theorem,
     ensuring the existence of
     a unique function $u_f\in \mathbb{X}(\Omega)$ such that
     \begin{equation} \label{eq.BufF}
      \text{$B(u_f,v) = F(v)$ for every $v\in \mathbb{X}(\Omega)$},
      \end{equation}
     further satisfying the estimate
     \begin{eqnarray*}&&\|u_f\|_{\mathbb{X}(\Omega)}\leq \mathbf{c}\,\|F\|_{(\mathbb{X}(\Omega))^*}
     =\mathbf{c}\,\sup_{\|v\|_{\mathbb{X}(\Omega)}=1}|F(v)|\le\mathbf{c}\,\sup_{\|v\|_{\mathbb{X}(\Omega)}=1}\int_{\Omega}|f|\,|v|\,\d x\\&& \le\mathbf{c}\,\sup_{\|v\|_{\mathbb{X}(\Omega)}=1}
    \|f\|_{L^2(\Omega)}\|v\|_{L^2(\Omega)}
     \leq  \mathbf{c}\,\|f\|_{L^2(\Omega)},\end{eqnarray*}
with the positive constant~$\mathbf{c}$ possibly varying from line to line
(actually, this argument shows that~$\|u_f\|_{\mathbb{X}(\Omega)}\leq \mathbf{c}\|f\|_{H^{-1}(\Omega)}$).
     {F}rom~\eqref{eq.BufF} we immediately conclude that $u_f$ is a weak solution of~\eqref{eq.mainPDE} (according to Definition~\ref{def.weaksol}),
     and the proof is complete.
    \end{proof}
    
    Since one of the aims of this paper is to
   prove regularity results for weak
   solution of~\eqref{eq.mainPB}, it is convenient to fix
   the following definition. We set
   \begin{equation} \label{eq.defSpaceLLs}
    \mathcal{C}_s(\RN) := \bigg\{u\in C(\RN,\R):\,\int_{\RN}\frac{|u(x)|}{1+|x|^{N+2s}}\,\d x
    <\infty\bigg\}.
   \end{equation}
   
   \begin{definition} \label{def.classicalsol}
    Let $f:\Omega\to\R$. We say that a function
    $u:\RN\to\R$ is a \emph{classical solution}
    of~\eqref{eq.mainPB} if it satisfies the following properties:
    \begin{itemize}
     \item[(1)] $u\in \CCs(\RN)\cap C^2(\Omega,\R)$;
     \item[(2)]
    $\LL u = f$ pointwise in $\Omega$; 
    \item[(3)] $u(x) = 0$ for every $x\in\RN\setminus\Omega$.
    \end{itemize} 
    If $u$ fulfills only (1) and (2), we say that $u$ is a \emph{classical
    solution} of~\eqref{eq.mainPDE}.
   \end{definition}
   We observe that, if $u\in\CCs(\RN)\cap C^2(\Omega,\R)$
   is a classical solution of~\eqref{eq.mainPB},
   it follows from (1) and (3) in Definition~\ref{def.classicalsol} that
   \begin{equation} \label{eq.uclassicalbounded}
    u\in L^\infty(\RN),\qquad\text{with\,\,\,$\|u\|_{L^\infty(\RN)}
    = \max_{x\in\overline{\Omega}}|u(x)|$.}
   \end{equation}
   
   \begin{remark} \label{rem.spaceCCs}
   As it is very well-known, the (linear) space
   $\CCs(\RN)$ is `good' for dealing with the fractional
   Laplacian. In fact, if $u\in C^2(\Omega,\R)\cap \CCs(\RN)$, it is possible
   to compute $(-\Delta)^s u(x)$ pointwise for every $x\in\Omega$, and
   \begin{align*}
    (-\Delta)^su(x) & = c_{N,s}\;\mathrm{P.V.}\int_{\RN}\frac{u(x)-u(y)}{|x-y|^{N+2s}}\,\d y
    \\[0.1cm]
    & = 
   -\frac{c_{N,s}}{2}\int_{\RN}\frac{u(x+z)+u(x-z)-2u(x)}{|z|^{N+2s}}\,\d z.
   \end{align*}
   Moreover, it is easy to check that
   $(-\Delta)^su\in L^\infty_{\loc}(\Omega)$.
   \end{remark}
   The next remark describes the relation between weak and classical
   solutions.
    \begin{remark} \label{rem.classical}
      Let $f\in L^2(\Omega)$ and let
      $u\in H^1(\RN)$ be a weak solution of~\eqref{eq.mainPDE}.
      If we further assume that
      $u\in \CCs(\RN)\cap C^2(\Omega,\R)$, we can compute
      $$\LL u(x) = -\Delta u(x) + (-\Delta)^s u(x)\quad\text{for every
      $x\in\Omega$}.$$
      Then, it is not difficult to check that $\LL u = f$ pointwise in $\Omega$,
      so that $u$ is a classical solution of~\eqref{eq.mainPDE}.
      Conversely, if $u\in \CCs(\RN)\cap C^2(\Omega,\R)$ is a classical
      solution of~\eqref{eq.mainPDE} further satisfying
      $u\in H^1(\RN)$, then $u$ is also a weak solution of~\eqref{eq.mainPDE}.
    \end{remark}
    
A simple consequence of Theorem~\ref{thm.existenceLax} is the solvability 
of the non\--ho\-mo\-ge\-neous Di\-ri\-chlet pro\-blem, as follows:
 
    \begin{corollary} \label{rem.solvabilitynonhom}
     Let $f\in L^2(\Omega)$ and $g\in \CCs(\RN)\cap C^2(\overline{\Omega},\R)\cap H^1(\RN)$.
     Then, there exists a unique weak solution $u\in H^1(\RN)$ of
     the non-homogeneous Dirichlet problem
     \begin{equation} \label{eq.PBnonhomog}
      \begin{cases}
    \LL u = f & \text{in $\Omega$}, \\
    u\big|_{\RN\setminus \Omega} = g.
   \end{cases}
   \end{equation}
   \end{corollary}
   
   \begin{proof}
   We observe that, since~$g\in \CCs(\RN)\cap C^2(\overline{\Omega},\R)$, the function $\LL g
     = -\Delta g+(-\Delta)^s g$ can be computed pointwise in
     $\Omega$, and
          $$\LL g\in L^\infty(\Omega)\subseteq L^2(\Omega).$$
 Hence, we can apply Theorem~\ref{thm.existenceLax} to get the existence
 of a unique weak solution~$v\in H^1(\RN)$
   of
     $$
     \begin{cases}
    \LL v = f - \LL g & \text{in $\Omega$}, \\
    v\big|_{\RN\setminus \Omega} = 0.
   \end{cases}$$
    Setting $u := v+g$, it is then immediate to see that
   $u\in H^1(\RN)$ solves~\eqref{eq.PBnonhomog}, thus completing the proof of
   Corollary~\ref{rem.solvabilitynonhom}.
    \end{proof}
    
    \begin{remark} \label{rem.spazioGiusto}
    Another functional framework naturally arising with
 the mixed operator $\LL$ is given by the spaces $\mathcal{H} := H^1(\Omega)\cap H^s(\RN)$
    and 
    $$\mathcal{H}_0 := 
    \big\{u\in\mathcal{H}:\,\text{$u\equiv 0$ a.e.\,in $\RN\setminus\Omega$}\big\}.$$
    In principle, one is tempted to define
a \emph{$\mathcal{H}$-weak solution of~\eqref{eq.mainPDE}}
as a function $u\in\mathcal{H}$ such that~\eqref{eq.weaksoldef}
    holds for every $\varphi\in C_0^\infty(\Omega,\R)$, and 
    a \emph{$\mathcal{H}_0$-weak solution of~\eqref{eq.mainPB}} as a $\mathcal{H}$-weak solution
    of~\eqref{eq.mainPDE} which belongs to $\mathcal{H}_0$.
    
    The use of Lax-Milgram's theorem would provide the existence of
	a unique $\mathcal{H}_0$-solution $v_f$ of~\eqref{eq.mainPB} (for
	some $f\in L^2(\Omega)$). Nevertheless, we prefer to use the functional
	setting in~\eqref{eq.defXOmega}, since it possesses better density properties
	for smooth compactly supported functions, allowing us to use the positive and negative parts
	of the solution as a test function in~\eqref{eq.weaksoldef}.
	
	This technical aspect is crucial for us in proving a weak maximum principle as in 
 Theorem~\ref{thm.WMPweak}.
 
 Another possible functional setting consists in requiring that, 
 for weak solutions,
	identity~\eqref{eq.weaksoldef} is fulfilled \emph{for all $v\in\mathcal{H}_0$}.
	Nevertheless, this approach 
causes a technical difficulty in the integration by parts formula, which is needed
	to prove that a classical solution
	of~\eqref{eq.mainPDE}-\eqref{eq.mainPB} is also a weak
	solution, thus confirming that our choice in~\eqref{eq.defXOmega} is likely to be
	structurally more robust to deal with the PDE properties of the solutions.
    \end{remark}
  
  \section{Some maximum principles for $\LL$} \label{sec.MPsLL}
  
  In this section we establish some weak/strong maximum principles
  for $\LL$, as defined in~\eqref{LOPERATO1}.
  To begin with, we prove the weak maximum principle for {\emph{weak solutions} in
  Theorem~\ref{thm.WMPweak}.

  \begin{proof}[Proof of Theorem~\ref{thm.WMPweak}]
   Arguing by contradiction, we assume that there exists
   a set $E\subseteq\Omega$, with \emph{positive Lebesgue measure},
   such that $u < 0$ a.e.\,on $E$. We then define
   $$w := u_- = \max\{-u,0\}$$
   and we observe that, since $u\in H^1(\RN)$ and $u\geq 0$ a.e.\,in $\RN\setminus\Omega$,
   one has 
\begin{equation}\label{if9ut9y49y}
\text{$w\in H^1(\RN)$ and $w\equiv 0$ a.e.\,in $\RN\setminus\Omega$}.\end{equation}
   Moreover,
   \begin{equation}\label{if9ut9y49y2}
   w = -u > 0 \qquad {\mbox{on $E$}}.
   \end{equation}
   In particular, recalling formula~\eqref{eq.defXOmega},
   from~\eqref{if9ut9y49y} we have that~$w\in\spX$.
   Hence, according to formula~\eqref{eq.weaksoldefDensity}
   in Remark~\ref{rem.WeakSoluf}, we can use $w$ as a test function
   in~\eqref{QUESTA}, obtaining that
    \begin{equation} \label{eq.dacontraddire}\begin{split}
    0&\le\; 
    \int_{\Omega}
     \langle\nabla u,\nabla w\rangle\,\d x
   +\int_{\RN\times\RN}
     \!\!\!\!\!\!\!\frac{(u(x)-u(y))(w(x)-w(y))}{|x-y|^{N+2s}}\,\d x\,\d y\\&=\;
     -\int_{\Omega}
     |\nabla  u_-|^2\,\d x
   +\int_{\RN\times\RN}
     \!\!\!\!\!\!\!\frac{(u(x)-u(y))(w(x)-w(y))}{|x-y|^{N+2s}}\,\d x\,\d y\\&\le
     \int_{\RN\times\RN}
     \!\!\!\!\!\!\!\frac{(u(x)-u(y))(w(x)-w(y))}{|x-y|^{N+2s}}\,\d x\,\d y
      .\end{split}
   \end{equation}
   
On the other hand, 
   denoting by $v$ the positive part of $u$, that is~$v := u_+ = \max\{u,0\}$,
   we have that~$u=v-w$. Therefore, utilizing~\eqref{if9ut9y49y} and~\eqref{if9ut9y49y2}, we have that
   \begin{equation}\begin{split}\label{u8547684yy8t}
     & \int_{\RN\times\RN}\frac{\big(u(x)-u(y)\big)
   \big(w(x)-w(y)\big)}{|x-y|^{N+2s}}\,\d x\,\d y \\
   & \qquad = 
   \int_{\RN\times\RN}\frac{\big(v(x)-v(y)\big)
   \big(w(x)-w(y)\big)}{|x-y|^{N+2s}}\,\d x\,\d y -
   \int_{\RN\times\RN}\frac{|w(x)-w(y)|^2}{|x-y|^{N+2s}}\,\d x\,\d y 
   \\
   & \qquad \leq 
   \int_{\RN\times\RN}\frac{\big(v(x)-v(y)\big)
   \big(w(x)-w(y)\big)}{|x-y|^{N+2s}}\,\d x\,\d y -
   \int_{E\times(\RN\setminus\Omega)}
   \frac{|w(x)-w(y)|^2}{|x-y|^{N+2s}}\,\d x\,\d y
   \\
   & \qquad
   =
   \int_{\RN\times\RN}\frac{\big(v(x)-v(y)\big)
   \big(w(x)-w(y)\big)}{|x-y|^{N+2s}}\,\d x\,\d y -
   \int_{E\times(\RN\setminus\Omega)}
   \frac{|w(x)|^2}{|x-y|^{N+2s}}\,\d x\,\d y
    \\
   & \qquad
   < \int_{\RN\times\RN}\frac{\big(v(x)-v(y)\big)
   \big(w(x)-w(y)\big)}{|x-y|^{N+2s}}\,\d x\,\d y.
   \end{split}
   \end{equation}
	Also, a `case-by-case' computation shows that
   $$(v(x)-v(y))(w(x)-w(y)) = (u_+(x)-u_+(y))(u_-(x)-u_-(y))
   \leq 0$$
   for almost every $x,y\in\RN$. Plugging this information into~\eqref{u8547684yy8t},
   we obtain that
   $$\int_{\RN\times\RN}
     \!\!\!\!\!\!\!\frac{(u(x)-u(y))(w(x)-w(y))}{|x-y|^{N+2s}}\,\d x\,\d y
     < 0.$$
   This is in contradiction with~\eqref{eq.dacontraddire}, and thus~$u\geq 0$ a.e.\,on $\Omega$, as desired.
  \end{proof}
  
 As regards \emph{classical solutions}, we now prove
 Theorem~\ref{NUOVO}:

 \begin{proof}[Proof of Theorem~\ref{NUOVO}]
First of all, we establish the weak maximum principle in~\eqref{thm.WMPLL}.
For this, we suppose by contradiction that there exists a point $\xi\in\Omega$
such that $u(\xi) < 0$. 
Since $u\in C(\RN,\R)$ and
  $\overline{\Omega}$ is compact, we can thus find
  $x_0\in\overline{\Omega}$ such that
  \begin{equation} \label{eq.defpointx0}
   u(x_0) = \min_{\overline{\Omega }}u < 0.
   \end{equation}
Moreover, since~$u\geq 0$ in~$\R^N\setminus\Omega$, we necessarily
  have that $x_0\in\Omega$, that is $x_0$ in an \emph{interior}
  minimum point for $u$ in $\Omega$. Hence,
  $\Delta u(x_0) \geq 0$ and
  \begin{equation} \label{eq.tocontradict}
   (-\Delta)^s u(x_0) ={\mathcal{L}}u(x_0)+ \Delta u(x_0)
  \geq \Delta u(x_0)\geq 0.
  \end{equation}
Moreover, we have that~$u(x_0)\leq u(x)$ for all~$x\in\R^N$.
Thus, we get
  \begin{equation} \label{eq.touseintegrand}
  \begin{split}
   (-\Delta)^su(x_0) & = 
  c_{N,s}\;\mathrm{P.V.}\int_{\RN}\frac{u(x_0)-u(y)}{|x_0-y|^{N+2s}}\,\d y 
  = c_{N,s}\int_{\RN}\frac{u(x_0)-u(y)}{|x_0-y|^{N+2s}}\,\d y \leq 0.
  \end{split}
  \end{equation}
 Owing to~\eqref{eq.tocontradict}, and taking into account the fact
 the integrand function in~\eqref{eq.touseintegrand} is non-positive, we then conclude that
 $$\text{$u\equiv u(x_0)$ on $\RN$}.$$ 
 In particular, since $u\geq 0$ in $\RN\setminus\Omega$,
 we get $u(x_0)\geq 0$, but this is clearly in contradiction with~\eqref{eq.defpointx0}.
 Hence, $u\geq 0$ on $\Omega$, which establishes~\eqref{thm.WMPLL}, as desired.

Now we prove the 
strong maximum principle for classical solutions in~\eqref{thm.SMPLL}.
To this end, we recall~\eqref{eq.tocontradict} and we note that
  \begin{equation}\label{eq.tocontradictE}
0\le (-\Delta)^su(x_0) = c_{N,s}\;\mathrm{P.V.}\int_{\RN}
  \frac{u(x_0)-u(y)}{|x_0-y|^{N+2s}}\,\d y  = -c_{N,s}\int_{\RN}\frac{u(y)}{|x_0-y|^{N+2s}}\,\d y  .\end{equation}
Additionally, by~\eqref{thm.WMPLL}, we know that~$u\ge0$ in~$\RN$. 
Comparing this with~\eqref{eq.tocontradictE}, we obtain that~$u\equiv 0$ throughout $\RN$. This establishes~\eqref{thm.SMPLL} as desired.
 \end{proof}
  We shall see in Appendix~\ref{sec.appendix} that, 
  due to the presence of the non-local term $(-\Delta)^s$, 
  a maximum principle analogous to that in \eqref{thm.SMPLL}
  \emph{does not hold} if the condition $u\geq 0$ is satisfied only
  on $\de\Omega$ (as in the classical case).
  Furthermore, we shall also show that
  maximum principles analogous to
  Theorem~\ref{thm.WMPweak} and in~\eqref{thm.WMPLL} of Theorem~\ref{NUOVO}
  \emph{do not hold}
  for
  $$\LL' := \Delta + (-\Delta)^s.$$

As a simple consequence of the weak maximum principle in~\eqref{thm.WMPLL}
we have a uniqueness result, as follows:

  \begin{corollary} \label{cor.uniqueClassical}
  There exists \emph{at most one} classical solution
  of~\eqref{eq.mainPB}.
 \end{corollary}
 \begin{proof}
  Let $u_1,u_2\in \CCs(\RN)\cap C^2(\Omega,\R)$ be two classical
  solutions of problem~\eqref{eq.mainPB}. Introducing the function
  $v:= u_1-u_2$, it is immediate to recognize that \medskip
  
  (a)\,\,$v\in \CCs(\RN)\cap C^2(\Omega,\R)$; \medskip
  
  (b)\,\,$v$ is a solution of the Dirichlet problem
  $$
   \begin{cases}
   \LL u = -\Delta u + (-\Delta)^s u = 0 & \text{pointwise in $\Omega$}, \\
   u = 0 & \text{in $\RN\setminus\Omega$}.
  \end{cases}
  $$
  Thus, by the weak maximum principle in~\eqref{thm.WMPLL} we readily conclude that $v\equiv 0$
  on $\RN$, so that~$u_1\equiv u_2$ on $\RN$. 
  This ends the proof.
 \end{proof}

\section{Interior and boundary regularity for $\LL$} \label{sec.regularity}

The main aim of this section is to prove
   both interior and boundary regularity
   for $\LL$. To be more precise, we
   first establish interior $H^m$-regularity
   for the weak solutions of~\eqref{eq.mainPDE}, that is
   Theorem~\ref{thm.higherregul};
   then, we prove boundary regularity for
   the solutions of~\eqref{eq.mainPB},
   that is Theorem~\ref{BDTH}.
   
   \subsection{Interior $H^m$-regularity and proof of Theorem~\ref{thm.higherregul}} \label{subsec.interior}
    To begin with, we prove the fol\-low\-ing~$H^2$\--re\-gu\-la\-ri\-ty 
    theorem, which in turn will serve as
    the basic step to prove interior $H^m$-regularity in Theorem~\ref{thm.higherregul}.
    
    \begin{theorem} \label{thm.mainregul}
     Let $f\in L^2(\Omega)$ and let $u\in H^1(\RN)$ be a weak solution of
     equation~\eqref{eq.mainPDE}. Then, $u\in H^2_{\loc}(\Omega)$.
     Furthermore, given any open set $V$ with $\overline{V}
     \subseteq\Omega$, there exists a constant $\Lambda > 0$, independent of
     $u$, such that
     \begin{equation} \label{eq.estimregulH2}
      \|u\|_{H^2(V)} \leq \Lambda\Big(\|f\|_{L^2(\Omega)}+\| u\|_{H^1(\RN)}\Big).
     \end{equation}
    \end{theorem}
    \begin{proof}
     Let $V$ be a fixed open set with $\overline{V}\subseteq\Omega$,
     and let $\rho_0 > 0$ be such that
     $$V_{\rho}:=\{x\in\RN:\,\mathrm{dist}(x,V) < \rho\}\subseteq\Omega
     \quad\text{for every $\rho\in[0,2\rho_0]$}.$$
     Moreover, let $\zeta\in C_0^\infty(\RN,\R)$
     be a cut-off function
     satisfying \medskip
     
     (a)\,\,$\zeta \equiv 1$ on $V$ and 
     $\mathrm{supp}(\zeta)\subseteq V_{\rho_0}$; \vspace{0.1cm}
     
     (b)\,\,$0\leq\zeta\leq 1$ on $\RN$. \medskip
     
     \noindent Finally, for every fixed $k\in\{1,\ldots,N\}$
     and every $0<|h|<\rho_0$, we set
     \begin{equation}\label{diuewt73pooii}
     \varphi := D_k^{-h}\big(\zeta^2\,D^h_k u\big),\qquad\text{where
     $D^h_k w(x) := \frac{w(x+he_k)-w(x)}{h}$}.\end{equation}
     We notice that, since~$u\in H^1(\RN)$, then~$D^h_k u\in H^1(\RN)$.
     Also, in light of~(a), we have that~$\zeta\in C_0^\infty(V_{\rho_0},\R)$.
As a result,
     $$\hat{\varphi} := \zeta^2\,D_k^h u\in H^1(\RN)\qquad\text{and}\qquad
     \text{supp}(\hat\varphi)\subseteq V_{\rho_0}\subseteq\Omega.$$
     As a consequence, by the definition of~$\varphi$ in~\eqref{diuewt73pooii},
     we see that~$\varphi\in H^1(\RN)$
     and $\mathrm{supp}(\varphi)\subseteq V_{2\rho_0}\subseteq\Omega$,
     which implies that~$\varphi\in H_0^1(\Omega)$.
     Therefore, we are in the position of
     using~$\varphi$
     as a test function in~\eqref{eq.weaksoldefDensity}, obtaining
     (after a standard `integration by parts'
     for difference quotients)
     \begin{equation} \label{eq.estimEvans1}
     \begin{split}
      & \sum_{i = 1}^N\int_{\Omega}\zeta^2\,
      |D^h_k(\de_{x_i}u)|^2\,\d x
      + \sum_{i = 1}^N  2\int_{\Omega}
      \zeta\,D^h_k u\,\de_{x_i}\zeta\,D^h_k(\de_{x_i}u)\,\d x
      \\[0.1cm]
      & \qquad
      + \frac{c_{N,s}}{2}\int_{\RN\times\RN}
      \frac{\big(D^h_k u(x)-D^h_k u(y)\big)\big(\zeta^2(x)D_k^h u(x)-
      \zeta^2(y)D_k^h u(y)\big)}{|x-y|^{N+2s}}\,\d x\,\d y \\[0.1cm]
      & \qquad 
      = \int_{\Omega}f\varphi\,\d x.
      \end{split}
     \end{equation}
     Now, by exploiting Cauchy-Swcharz's inequality
     and the classical Young inequality (with $\varepsilon = 1/2$), we 
     obtain the following estimate
     \begin{equation} \label{eq.estimEvans2}
     \begin{split}
	& \bigg|\sum_{i = 1}^N  2\int_{\Omega}
      \zeta\,D^h_k u\,\de_{x_i}\zeta\,D^h_k(\de_{x_i}u)\,\d x\bigg|
      \leq 2\int_{\Omega}\zeta\,|D^h_k u|\,|\nabla \zeta|\,|D_k^h(\nabla u)|\,\d x \\[0.1cm]
      & \qquad\leq
      \frac{1}{2}\int_{\Omega}\zeta^2\,|D_k^h(\nabla u)|^2\,\d x
      + 2\,\big(\sup_{\RN}|\nabla\zeta|\big)\cdot\int_{V_{\rho_0}}|D^h_k u|^2\,\d x \\[0.1cm]
      & \qquad \leq 
      \frac{1}{2}\int_{\Omega}\zeta^2\,|D_k^h(\nabla u)|^2\,\d x
      + C\,\int_{\Omega}|\nabla u|^2\,\d x,
      \end{split}
     \end{equation}
     for some~$C>0$,
     where in the last inequality we have used the fact that
     \begin{equation} \label{eq.estimEVANSdff}
      \int_{\RN}|D^h_k \omega|^2\,\d x\leq
     4N^2\int_{\RN}|\nabla \omega|^2\,\d x, \qquad\text{for every $\omega\in H^1(\RN)$}
     \end{equation}
     (see, e.g.,~\cite[Theorem.\,3, Chapter\,5.8.2]{Evans}).
     Gathering together estimates~\eqref{eq.estimEvans1}
     and~\eqref{eq.estimEvans2}, we get
     \begin{equation} \label{eq.EvansStepIII}
     \begin{split}
      & \frac{1}{2}\int_{\Omega}\zeta^2\,|D_k^h(\nabla u)|^2\,\d x
      - C \,\int_{\Omega}|\nabla u|^2\,\d x \\[0.1cm]
      & \qquad
      + \frac{c_{N,s}}{2}\int_{\RN\times\RN}
      \frac{\big(D^h_k u(x)-D^h_k u(y)\big)\big(\zeta^2(x)D_k^h u(x)-
      \zeta^2(y)D_k^h u(y)\big)}{|x-y|^{N+2s}}\,\d x\,\d y \\[0.1cm]
      & \qquad 
      \le \int_{\Omega}f\varphi\,\d x.
      \end{split}
     \end{equation}
     
     Now we estimate the integral in the right hand side
     of~\eqref{eq.EvansStepIII}. To this end we first observe that, recalling~\eqref{diuewt73pooii}
     and exploiting once again~\eqref{eq.estimEVANSdff}, we have
     \begin{equation*}
      \begin{split}
       & \int_{\Omega}\varphi^2\,\d x 
       = \int_{\RN}\big|D^{-h}_k\big(\zeta^2 D_k^h u\big)\big|^2\,\d x\\&\qquad
       \leq 4N^2\,\int_{\RN}|\nabla (\zeta^2 D^h_k u)|^2\,\d x 
        = C\,\int_{V_{\rho_0}}|\nabla (\zeta^2 D^h_k u)|^2\,\d x \\
       & \qquad \leq C\bigg(\int_{V_{\rho_0}}|D_k^hu|^2\,\d x+\int_{\Omega}
       \zeta^2\,|D_k^h(\nabla u)|^2\,\d x\bigg) \\&\qquad \leq C
       \bigg(\int_{\Omega}|\nabla u|^2\,\d x
       + \int_{\Omega}\zeta^2\,|D_k^h(\nabla u)|^2\,\d x\bigg),
      \end{split}
     \end{equation*}
     up to renaming~$C>0$ from line to line.
     {F}rom this, using Young's inequality (with $\varepsilon = 1/(4C)$), we get
     \begin{equation} \label{eq.EvansStepIV}
      \begin{split}
      & \bigg|\int_{\Omega}f\varphi\,\d x\bigg|
      \leq \varepsilon\int_{\Omega}\varphi^2\,\d x+\frac{1}{\varepsilon} 
      \int_{\Omega}f^2\,\d x \\[0.1cm]
      & \qquad\leq
      \frac{1}{4}\int_{\Omega}|\nabla u|^2\,\d x
       + \frac{1}{4}\int_{\Omega}\zeta^2\,|D_k^h(\nabla u)|^2\,\d x
       +4C \int_{\Omega}f^2\,\d x.
      \end{split}
     \end{equation}
     By combining~\eqref{eq.EvansStepIV} with~\eqref{eq.EvansStepIII}, we then obtain
     \begin{equation} \label{eq.EvansStepV}
     \begin{split}
      & \frac{1}{4}\int_{\Omega}\zeta^2\,|D_k^h(\nabla u)|^2\,\d x \\[0.1cm]
      & \qquad
      +\frac{c_{N,s}}{2}\int_{\RN\times\RN}
      \frac{\big(D^h_k u(x)-D^h_k u(y)\big)\big(\zeta^2(x)D_k^h u(x)-
      \zeta^2(y)D_k^h u(y)\big)}{|x-y|^{N+2s}}\,\d x\,\d y \\[0.1cm]
      & \qquad \leq C\bigg(\int_{\Omega}f^2\,\d x
      +\int_{\Omega}|\nabla u|^2\,\d x\bigg),
      \end{split}
     \end{equation}
     up to relabeling~$C>0$.
  
     We now provide a careful estimate of
     the non-local term in the left hand side of~\eqref{eq.EvansStepV}, i.e,
\begin{equation}\label{h67546yghgvb}
J_s := \int_{\RN\times\RN}
      \frac{\big(D^h_k u(x)-D^h_k u(y)\big)\big(\zeta^2(x)D_k^h u(x)-
      \zeta^2(y)D_k^h u(y)\big)}{|x-y|^{N+2s}}\,\d x\,\d y.\end{equation}
     To this end we first notice that, with obvious algebraic
     manipulation, we can write
     \begin{equation} \label{eq.splitJs}
       J_s = J_{0,s} + J_{1,s}+J_{2,s},
      \end{equation}
      where
      \begin{equation*}\begin{split}
      J_{0,s}&\; :=
      \int_{\RN\times\RN}
      \zeta^2(x)\cdot\frac{|D_k^hu(x)-D_k^hu(y)|^2}{|x-y|^{N+2s}}\,\d x\,\d y,
      \\
      J_{1,s}&\, := \int_{\RN\times\RN}
      D_k^hu(y)\,\zeta(x)\cdot\frac{
      \big(D^h_k u(x)-D^h_k u(y)\big)\big(\zeta(x)-\zeta(y)\big)}{|x-y|^{N+2s}}\,\d x\,\d y \\
    {\mbox{and }}\quad J_{2,s}&\;:=
     \int_{\RN\times\RN}D_k^hu(y)\,\zeta(y)\cdot\frac{
      \big(D^h_k u(x)-D^h_k u(y)\big)\big(\zeta(x)-\zeta(y)\big)}{|x-y|^{N+2s}}\,\d x\,\d y.
      \end{split}
     \end{equation*}
     We notice that, by exchanging the variables~$x$ and~$y$, we see that
     \begin{equation}\begin{split}\label{848476y85y8}
     J_{0,s}\;&=
      \int_{\RN\times\RN}
      \zeta^2(y)\cdot\frac{|D_k^hu(y)-D_k^hu(x)|^2}{|y-x|^{N+2s}}\,\d x\,\d y\\
      \;&=
      \int_{\RN\times\RN}
      \zeta^2(y)\cdot\frac{|D_k^hu(x)-D_k^hu(y)|^2}{|x-y|^{N+2s}}\,\d x\,\d y.
    \end{split} \end{equation}
     Moreover, since $\zeta\in C_0^\infty(\RN,\R)$, for every $y\in\RN$ we have
     \begin{equation} \label{eq.estimateintegralpsi}
      \begin{split}
      & \int_{\RN}\frac{|\zeta(x)-\zeta(y)|^2}{|x-y|^{N+2s}}\,\d x\\
       \leq \;&\Big(\sup_{\RN}|\nabla \zeta|^2\Big)\cdot\int_{\{|x-y|\leq 1\}}
       \frac{\d x}{|x-y|^{N+2(s-1)}} +4\,\Big(\sup_{\RN}|\zeta|^2\Big)\cdot\int_{\{|x-y| > 1\}}
       \frac{\d x}{|x-y|^{N+2s}}
\\       =: \;&\mathbf{c}(\zeta,s).
     \end{split}
     \end{equation}
     Using Young's inequality once again, \eqref{eq.estimateintegralpsi} and~\eqref{eq.estimEVANSdff},
     we can estimate $J_{1,s}$ as follows:
     \begin{equation} \label{eq.estimJ1s}
      \begin{split}
       |J_{1,s}| \leq\;& \frac{1}{4}\,\int_{\RN\times\RN}
      \zeta^2(x)\,\cdot\frac{|D_k^hu(x)-D_k^hu(y)|^2}{|x-y|^{N+2s}}\,\d x\,\d y \\&\qquad
    + 4\int_{\RN\times\RN}
      |D_k^h u(y)|^2\cdot\frac{|\zeta(x)-\zeta(y)|^2}{|x-y|^{N+2s}}\,\d x\,\d y
      \\
            \leq \;&\frac{J_{0,s}}{4}
      + 4\,\mathbf{c}(\zeta,s)\,\int_{\RN}|D_k^h u(y)|^2\,\d y \\
      \leq\;& \frac{J_{0,s}}{4}+4\,\mathbf{c}(\zeta,s)\,\int_{\RN}|\nabla u|^2\,\d x.
      \end{split}
     \end{equation}
     Similarly, making again use of Young's inequality and~\eqref{eq.estimEVANSdff},
     and recalling~\eqref{848476y85y8},
     we have the following estimate for $J_{2,s}$:
     \begin{equation} \label{eq.estimJ2s}
      \begin{split}
        |J_{2,s}| \leq\;& \frac{1}{4}\,\int_{\RN\times\RN}
      \zeta^2(y)\,\cdot\frac{|D_k^hu(x)-D_k^hu(y)|^2}{|x-y|^{N+2s}}\,\d x\,\d y 
      \\
      & \qquad + 4\int_{\RN\times\RN}
      |D_k^h u(y)|^2\cdot\frac{|\zeta(x)-\zeta(y)|^2}{|x-y|^{N+2s}}\,\d x\,\d y
      \\
      \leq \;&\frac{J_{0,s}}{4}
      + 4\,\mathbf{c}(\zeta,s)\,\int_{\RN}|D_k^h u(y)|^2\,\d y \\
      \leq \;&\frac{J_{0,s}}{4}+4\,\mathbf{c}(\zeta,s)\,\int_{\RN}|\nabla u|^2\,\d x.
      \end{split}
     \end{equation}
     Gathering together~\eqref{eq.splitJs}, \eqref{eq.estimJ1s} and~\eqref{eq.estimJ2s}, we obtain that
     \begin{equation} \label{eq.estimJsfinal}
      J_s \geq \frac{J_{0,s}}{2} - 8\,\mathbf{c}(\zeta,s)\int_{\RN}|\nabla u|^2\,\d x.
     \end{equation}
     Finally, by combining~\eqref{eq.EvansStepV} and~\eqref{eq.estimJsfinal},
     and recalling~\eqref{h67546yghgvb},
     we derive
     \begin{align*}
      & \frac{1}{4}\int_{\Omega}\zeta^2\,|D_k^h(\nabla u)|^2\,\d x 
      + \frac{c_{N,s}}{2}\cdot\frac{J_{0,s}}{2} \\[0.1cm]
      & \qquad \leq
    4{c_{N,s}}\,\mathbf{c}(\zeta,s)\int_{\RN}|\nabla u|^2\,\d x
      + C\bigg(\int_{\Omega}f^2\,\d x
      +\int_{\Omega}|\nabla u|^2\,\d x\bigg) \\[0.1cm]
      & \qquad \leq {C}\bigg(\int_{\Omega}f^2\,\d x
      +\int_{{\RN}}|\nabla u|^2\,\d x\bigg),
     \end{align*}
     up to renaming~$C>0$.
     {F}rom this, since $J_{0,s}\geq 0$, we get
     \begin{equation*}
     \begin{split}
      \int_{V}|D_k^h(\nabla u)|^2\,\d x & \leq 
      \int_{\Omega}\zeta^2\,|D_k^h(\nabla u)|^2\,\d x \leq  {C}\bigg(\int_{\Omega}f^2\,\d x
      +\int_{\RN}|\nabla u|^2\,\d x\bigg),
      \end{split}
     \end{equation*}
     and thus, owing to~\cite[Theorem\,3, Chapter\,5.8.2]{Evans},
     we conclude that 
     $u\in H^2(V)$ and
     $$\|u\|_{H^2(V)}\leq \Lambda\Big(\|f\|_{L^2(\Omega)}+
     \|u\|_{H^1(\RN)}\Big),$$
     for a suitable $\Lambda > 0$. Finally, since
     a careful inspection of the proof 
     shows that $\Lambda$ does not depend on $u$,
     we also obtain estimate~\eqref{eq.estimregulH2}.
    \end{proof}
    
    Starting from Theorem~\ref{thm.mainregul}, our next aim
    it to establish Theorem~\ref{thm.higherregul}.
    In contrast to the classical case, the proof of Theorem~\ref{thm.higherregul} 
    is not merely a bootstrap argument based on Theorem~\ref{thm.mainregul}: in fact,
    the presence of the non-local term $(-\Delta)^s$ in $\LL$
    prevents us to take derivatives of
equation~\eqref{eq.mainPDE}. 
    To overcome this technical issue, we need to combine
    a suitable truncation argument with the use of difference quotient.
    This is done in the following preliminary lemmata. \medskip
    
    In the sequel, we shall make use of the following notation:
    given an arbitrary set $A\subseteq\RN$ and a number $\delta > 0$, we define
    \begin{equation}\label{05785987vufdhg}
    A_{\delta} := \big\{x\in\RN:\,\mathrm{dist}(x,A) < \delta\big\}.\end{equation}
    Moreover, for any $\omega\in L^2(\RN)$, we set
   \begin{equation}\label{f854yhgjgk}
   \omega_h(x) := D^h_k\omega(x) = \frac{\omega(x+he_k)-\omega(x)}{h},\end{equation}
    with $h\in\R$ and $k\in\{1,\ldots,N\}$.
    
    \begin{lemma} \label{lem.rapprinc}
     Let $u\in H^1(\RN)$ be a solution of~\eqref{eq.mainPDE},
     and let $\OO$ be an open set with
     $\overline{\OO}\subseteq\Omega$. 
     Let~$\rho := \mathrm{dist}(\OO,\de\Omega) > 0$.
    If $|h|<\rho$, then $u_h$ solves
     \begin{equation} \label{eq.solvedbyuh}
      \LL u_h = f_h\quad \text{in $\OO$}.
      \end{equation}
    \end{lemma}
    
    \begin{proof}
 We observe that, for any~$\varphi\in C_0^\infty(\OO,\R)$,
     we have that
     \begin{equation} \label{eq.identitylocalDu}
     \begin{split}
      & \int_{\OO}\langle \nabla u_h, \nabla \varphi\rangle\,\d x = 
      \int_{\RN}\langle \nabla u_h, \nabla \varphi\rangle\,\d x
       = -\int_{\RN}\langle \nabla u, \nabla \psi_{-h}\rangle\,\d x,
      \end{split}
     \end{equation}
     where we have used the notation
\begin{equation}\label{fh437ytbvjdbvfksanvckv}
\psi_{-h}(x) := -\frac{\varphi(x-he_k)-\varphi(x)}{h}.\end{equation}
Moreover, since $\mathrm{supp}(\varphi)\subseteq\OO$, recalling the
notation in~\eqref{05785987vufdhg}, we have that
 $$\mathrm{supp}(\varphi(\cdot-he_k))\subseteq
     \OO+he_k\subseteq \big\{x\in\RN:\,\mathrm{dist}(x,\OO)<|h|\big\}
     = \OO_{|h|}.$$
  Hence, if $|h|<\rho$ we obtain that \medskip
     
     (a)\,\,$\psi_{-h}\in C_0^\infty(\RN,\R)$; \medskip
     
     (b)\,\,
     $
     \mathrm{supp}(\psi_{-h}) \subseteq \OO_{\rho}\subseteq\Omega$. \medskip
     
    \noindent In particular,
	\begin{equation}\label{9wfhighhlke}
	\psi_{-h}\in C_0^\infty(\OO_{\rho},\R).
	\end{equation}
     As a consequence, identity~\eqref{eq.identitylocalDu} can be written as
     \begin{equation} \label{eq.identitylocalDu2}
      \int_{\OO}\langle \nabla u_h, \nabla \varphi\rangle\,\d x = 
      -\int_{\Omega}\langle \nabla u, \nabla \psi_{-h}\rangle\,\d x.
     \end{equation}
     Furthermore, we observe that, for all $h\in\R$, we have
     \begin{equation} \label{eq.identityNONlocal}
      \begin{split}
       & \int_{\RN\times\RN}
     \frac{(u_h(x)-u_h(y))(\varphi(x)-\varphi(y))}{|x-y|^{N+2s}}\,\d x\,\d y \\[0.1cm]
     & \qquad = -\int_{\RN\times\RN}
     \frac{(u(x)-u(y))(\psi_{-h}(x)-\psi_{-h}(y))}{|x-y|^{N+2s}}\,\d x\,\d y.
      \end{split}
     \end{equation}
     Thus, by combining~\eqref{eq.identitylocalDu2} with~\eqref{eq.identityNONlocal}, 
     and recalling~\eqref{9wfhighhlke} and the fact that~$u$ solves~\eqref{eq.mainPDE},
      we get
     \begin{equation}\label{85ythgdikkkhgkelgoe}
 	 \begin{split}
 	 & \int_{\OO}
 	 \langle \nabla u_h, \nabla \varphi\rangle\,\d x
 	 + \frac{c_{N,s}}{2}\int_{\RN\times\RN}
     \frac{(u_h(x)-u_h(y))(\varphi(x)-\varphi(y))}{|x-y|^{N+2s}}\,\d x\,\d y \\[0.1cm]
     & \quad
     = -\int_{\Omega}\langle \nabla u, \nabla \psi_{-h}\rangle\,\d x
     -\frac{c_{N,s}}{2}\int_{\RN\times\RN}
     \frac{(u(x)-u(y))(\psi_{-h}(x)-\psi_{-h}(y))}{|x-y|^{N+2s}}\,\d x\,\d y \\[0.1cm]
     & \qquad = -\int_{\OO_{\rho}}f\,\psi_{-h}\,\d x. 
 	 \end{split}     
     \end{equation}    
     Now we observe that, if $|h|<\rho$,
     $$\mathrm{supp}(\varphi)\subseteq\OO
     \subseteq\OO_{\rho}\cap(\OO_{\rho}-he_k),$$
     and therefore, recalling~\eqref{f854yhgjgk} and~\eqref{fh437ytbvjdbvfksanvckv},
    we have that
      \begin{align*}
     -\int_{\OO_{\rho}}f\,\psi_{-h}\,\d x
      =\;& \int_{\OO_{\rho}}f(x)\bigg(\frac{\varphi(x-he_k)-\varphi(x)}{h}\bigg)\d x
     \\
    = \;&\frac{1}{h}\bigg(\int_{\OO_{\rho}-he_k}f(x+he_k)\varphi(x)\,\d x
     -\int_{\OO_{\rho}}f(x)\,\varphi(x)\,\d x\bigg) \\
     =\;& \int_{\OO}f_h\,\varphi\,\d x.
     \end{align*}
     As a consequence of this and~\eqref{85ythgdikkkhgkelgoe},
     $$  \int_{\OO}
 	 \langle \nabla u_h, \nabla \varphi\rangle\,\d x
 	 + \frac{c_{N,s}}{2}\int_{\RN\times\RN}
     \frac{(u_h(x)-u_h(y))(\varphi(x)-\varphi(y))}{|x-y|^{N+2s}}\,\d x\,\d y=
      \int_{\OO}f_h\,\varphi\,\d x,$$
      for any~$\varphi\in C_0^\infty(\OO,\R)$.
      Thus, recalling~\eqref{LOPERATO1}, this implies
     that $u_h$ solves~\eqref{eq.solvedbyuh}, as desired.
    \end{proof}
    
    Thanks to Lemma~\ref{lem.rapprinc}, we can prove a 
    `weaker version'
    of Theorem~\ref{thm.higherregul}.
    
    \begin{proposition} \label{prop.higherregweak}
     Let $m\in\mathbb{N}\cup\{0\}$ and~$f\in H^m(\Omega)$. If $u\in H^{m+1}(\RN)$ is any weak solution
      of~\eqref{eq.mainPDE}, then
      $u\in H^{m+2}_{\loc}(\Omega)$. \vspace*{0.07cm}
      
      Furthermore, given any open set $V$ with 
      $\overline{V}\subseteq\Omega$, there exists a constant $\Lambda_m > 0$, independent of
     the function $u$, such that
     \begin{equation} \label{eq.estiminduttiva}
      \|u\|_{H^{m+2}(V)}\leq \Lambda_m\Big(\|f\|_{H^m(\Omega)}+
      \|u\|_{H^{m+1}(\RN)}\Big).
     \end{equation}
    \end{proposition}
    
    \begin{proof}
     We proceed by induction on $m\in\mathbb{N}\cup\{0\}$.
     First of all, the case $m = 0$ (that is, $f\in L^2(\Omega)$
     and $u\in H^1(\RN)$) 
     is given by Theorem~\ref{thm.mainregul}. Then, we assume that
   Proposition~\ref{prop.higherregweak} holds for some $m\geq 0$,
     and we prove that it still holds for 
     $m + 1$. \medskip
     
     Let $V$ be a fixed open set with $\overline{V}\subseteq\Omega$,
     and let $\OO$ be an open subset of $\Omega$ satisfying~$\overline{V}\subseteq\OO$
     and~$\overline{\OO}\subseteq\Omega$.
     Moreover, let $f\in H^{m+1}(\Omega)$ and let
     $u\in H^{m+2}(\RN)$ be a weak solution of~\eqref{eq.mainPDE}.
    
     Setting $\rho := \mathrm{dist}(\OO,\de\Omega) > 0$, 
     and recalling the notation in~\eqref{f854yhgjgk}, if~$|h|<\rho$
     we know from Lemma~\ref{lem.rapprinc} that
     $u_h$ is a weak solution of the equation
     $$\text{$\LL u_h = f_h$ in $\OO$}.$$
     Moreover, since~$u\in H^{m+2}(\RN)$, we clearly
     have that~$u_h\in H^{m+1}(\RN)$. As a result, we are in the position of applying
     the inductive hypothesis to $u_h$, thus obtaining \medskip
     
     (a)\,\,$u_h\in H^{m+2}_{\loc}(\OO)$ and, in particular,
     $u_h\in H^{m+2}(V)$; \medskip
     
     (b)\,\,there exists a constant $\Lambda_m > 0$, independent of $h$, such that
     $$\|u_h\|_{H^{m+2}(V)}\leq \Lambda_m\Big(\|f_h\|_{H^m(\OO)}+
      \|u_h\|_{H^{m+1}(\RN)}\Big) \qquad
    \text{for $|h|<\rho$}.$$
      
   Furthermore, we observe that, since $u\in H^{m+2}(\RN)$, one has
    $$\|u_h\|_{H^{m+1}(\RN)} \leq \mathbf{c}\,\|u\|_{H^{m+2}(\RN)},$$
    where $\mathbf{c} > 0$ is a suitable constant which
    is independent of $h$ (see, e.g.,~\cite{Evans}). Analogously,
    since $f\in H^{m+1}(\Omega)$, we also have
    $$\|f_h\|_{H^{m}(\OO)}\leq \mathbf{c}\,\|f\|_{H^{m+1}(\Omega)}.$$
    Gathering together these facts, we obtain
    $$\|u_h\|_{H^{m+2}(V)}\leq 
    \mathbf{c}\,\Lambda_m\Big(\|f\|_{H^{m+1}(\Omega)}
    +\|u\|_{H^{m+2}(\RN)}\Big),$$
    and this estimate is \emph{uniform} with respect
    to $h\in(-\rho,\rho)$. On account of~\cite[Theorem.\,3, Chapter\,5.8.2]{Evans},
    we then easily conclude that
    $u\in H^{m+3}(V)$. Moreover,
    $$\|u\|_{H^{m+3}(V)}\leq \mathbf{c}\,
    \Lambda_m\Big(\|f\|_{H^{m+1}(\Omega)}
    +\|u\|_{H^{m+2}(\RN)}\Big).$$
    This is precisely estimate~\eqref{eq.estiminduttiva}, and the proof is complete.
    \end{proof}
    
    To remove the assumption that~$u\in H^{m+1}(\RN)$ in Proposition~\ref{prop.higherregweak}, 
    we need to perform a ``truncation'' argument: this is described
    in the next two lemmata.
    
    \begin{lemma} \label{lem.generalII}
	Let $\mathcal{O}\subseteq\RN$ be open,
	$\delta > 0$, and~$\alpha > N$. Let $z\in L^2(\RN)$ be such that
	\begin{equation}\label{54y544uuu}
	{\mbox{$z\equiv 0$ a.e.\,on $\mathcal{O}_\delta$,}}
	\end{equation}
	 with the notation introduced in~\eqref{05785987vufdhg}. Then, the following facts hold.
	\begin{itemize}
	 \item[{(i)}] for every fixed $x\in\mathcal{O}_{\delta/2}$, we have
	 \begin{equation} \label{eq.summabilitykernelz}
	  y\mapsto \frac{z(y)}{|x-y|^\alpha}\in L^1(\RN).
	 \end{equation}
	 \item[{(ii)}] The function $\II_{\alpha}[z]$ defined as
	 \begin{equation} \label{eq.smoothII}
	  \II_{\alpha}[z](x) :=
	  \int_{\RN}\frac{z(y)}{|x-y|^\alpha}\,\d y
	 \end{equation}
	 is of class $C^\infty$ on $\mathcal{O}_{\delta/2}$.
	\end{itemize}
    \end{lemma}
    \begin{proof}
     (i)\,\,Let $x\in\mathcal{O}_{\delta/2}$. {F}rom~\eqref{54y544uuu} and
     H\"older's inequality, we have the following estimate
     \begin{equation}\label{f487y54878467}
     \begin{split}
      \int_{\RN}\frac{|z(y)|}{|x-y|^\alpha}\,\d y
      & = \int_{\RN\setminus\mathcal{O}_\delta}\frac{|z(y)|}{|x-y|^\alpha}\,\d y  \\[0.1cm]
      & \leq 
      \bigg(\int_{\RN\setminus\mathcal{O}_\delta}\frac{|z(y)|^2}{|x-y|^\alpha}\,\d y\bigg)^{1/2}
      \cdot\bigg(\int_{\RN\setminus\mathcal{O}_\delta}\frac{1}{|x-y|^\alpha}\,\d y\bigg)^{1/2} .
      \end{split}
     \end{equation}
     On the other hand, since $x\in\mathcal{O}_{\delta/2}$, it is immediate to check that
     \begin{equation} \label{eq.mainestimkernelLemma}
      |x-y|\geq \frac{\delta}{2} \qquad\text{for every $y\in\RN\setminus\mathcal{O}_\delta$}.
     \end{equation}
     By exploiting~\eqref{eq.mainestimkernelLemma}, we get from~\eqref{f487y54878467} that
     \begin{eqnarray*}
      \int_{\RN}\frac{|z(y)|}{|x-y|^\alpha}\,\d y
      &\leq& \bigg(\frac{2}{\delta}\bigg)^{\alpha/2}
      \bigg(\int_{\{|x-y|\geq\delta/2\}}\frac{1}{|x-y|^\alpha}\,\d y\bigg)^{1/2}\cdot
      \|z\|_{L^2(\RN)} \\&\leq&
      \bigg(\frac{2}{\delta}\bigg)^{\alpha/2}
      \bigg(\int_{\{|w|\geq\delta/2\}}\frac{1}{|w|^\alpha}\,\d w\bigg)^{1/2}\cdot
      \|z\|_{L^2(\RN)}.
     \end{eqnarray*}
     {F}rom this, reminding that $\alpha > N$ and that 
     $z\in L^2(\RN)$, we obtain~\eqref{eq.summabilitykernelz}. \medskip
     
     (ii)\,\,First of all, owing to (i),
     the function $\II_\alpha[z]$ is well-posed on $\mathcal{O}_{\delta/2}$.
     To prove its smoothness on $\mathcal{O}_{\delta/2}$
     we show that, for every
     fixed $x\in\mathcal{O}_{\delta/2}$ and every
     $N$-tuple $\gamma = 
     (\gamma_1,\ldots,\gamma_N)$ of non-negative integers, one has
     \begin{equation} \label{eq.derivativeII}
       \big(\de_{x_1}\big)^{\gamma_1}\cdots\big(\de_{x_N}\big)^{\gamma_N}
      \II_\alpha[z](x)    = \int_{\RN\setminus\mathcal{O}_{\delta}}
      \big(\de_{x_1}\big)^{\gamma_1}\cdots\big(\de_{x_N}\big)^{\gamma_N}
      \bigg(\frac{z(y)}{|x-y|^\alpha}\bigg)\d y.
     \end{equation}
	 To this end we first observe that, setting
	 $|\gamma| := \displaystyle\sum_{k = 1}^N\gamma_k$, one has
	 \begin{equation} \label{eq.boundwithpowerdist}
	 \bigg|\big(\de_{x_1}\big)^{\gamma_1}\cdots\big(\de_{x_N}\big)^{\gamma_N}
      \bigg(\frac{z(y)}{|x-y|^\alpha}\bigg)\bigg|
      \leq \kappa_\alpha\,\frac{|z(y)|}{|x-y|^{\alpha+|\gamma|}}
      \end{equation}
	 for every $x\neq y\in\RN$ (here, $\kappa_\alpha$ is a positive
	 constant only depending on $\alpha$). Thus, by assertion (i)
	 (applied to~$\alpha+|\gamma|$ in place of $\alpha$), we derive that
	 $$y\mapsto \mathcal{D}_x(y) := \big(\de_{x_1}\big)^{\gamma_1}
     \cdots\big(\de_{x_N}\big)^{\gamma_N}
      \bigg(\frac{z(y)}{|x-y|^\alpha}\bigg)
     \in L^1(\RN)$$
     for every fixed $x\in\mathcal{O}_{\delta/2}$.
     On account of this fact,
     and owing to
     classical results
     on the regularity of
     parameter-depending integrals,
     to establish~\eqref{eq.derivativeII} it suffices
     to prove the following fact: \emph{for every
     $N$-tuple $\gamma = (\gamma_1,\ldots\gamma_N)$ of non-negative integers
     and every $x_0\in\mathcal{O}_{\delta/2}$, there exist $r > 0$
     and a function
  \begin{equation}\label{yfewt8476874679}
  \Theta = \Theta_{\gamma,x_0,r}\in L^1(\RN\setminus\mathcal{O}_\delta)\end{equation}
     such that}
     \begin{itemize}
      \item[(a)] $B(x_0,r)\subseteq\mathcal{O}_{\delta/2}$;
      \item[(b)] for every $x\in B(x_0,r)$ and every
      $y\in\RN\setminus\mathcal{O}_\delta$, one has
      \begin{equation} \label{eq.touseDerivativebound}
     \bigg|\big(\de_{x_1}\big)^{\gamma_1}\cdots\big(\de_{x_N}\big)^{\gamma_N}
      \bigg(\frac{z(y)}{|x-y|^\alpha}\bigg)\bigg|\leq \Theta(y).
     \end{equation}
    \end{itemize}
     To prove this statement, let $x_0\in\mathcal{O}_{\delta/2}$, and
     $\gamma = (\gamma_1,\ldots,\gamma_N)\in(\N\cup\{0\})^N$.
     Moreover, let $r > 0$ be such that $B(x_0,r)\subseteq\mathcal{O}_{\delta/2}$.
     We claim that there exists a constant $\mathbf{c} > 0$,
     only depending on $x_0$ and $r$, such that
     \begin{equation} \label{eq.boundKernelClaim}
      \frac{|x-y|}{|x_0-y|}\geq \mathbf{c}\qquad\text{for all
      $x\in B(x_0,r)$ and $y\in\RN\setminus\mathcal{O}_\delta$}. 
     \end{equation}
     Indeed, recalling~\eqref{eq.mainestimkernelLemma}, if $x\in B(x_0,r)\subseteq \mathcal{O}_{\delta/2}$
      and $y\in 
     (\RN\setminus\mathcal{O}_\delta)\cap B(x_0,2r)$, one has
     $$\frac{|x-y|}{|x_0-y|}
     \geq \frac{\delta/2}{2r} = \frac{\delta}{4r}.$$
  	On the other hand,  if~$y\in\big(\RN\setminus\mathcal{O}_\delta\big)\setminus B(x_0,2r)$, 
  	we have that
	$$|x-x_0|<r < \tfrac{1}{2}|x_0-y|,$$
	and therefore, by triangle inequality,
    \begin{align*}
     & \frac{|x-y|}{|x_0-y|} \geq 
     \frac{|x_0-y|-|x-x_0|}{|x_0-y|}
     = 1-\frac{|x-x_0|}{|x_0-y|} \geq \frac{1}{2}.
     \end{align*}
     Gathering together these facts, we obtain~\eqref{eq.boundKernelClaim} with 
     $$\mathbf{c} := \min\big\{\delta/(4r), 1/2\big\}.$$     
   Now, by combining~\eqref{eq.boundwithpowerdist} with~\eqref{eq.boundKernelClaim} we get
     \begin{align*}
     \bigg|\big(\de_{x_1}\big)^{\gamma_1}\cdots\big(\de_{x_N}\big)^{\gamma_N}
      \bigg(\frac{z(y)}{|x-y|^\alpha}\bigg)\bigg|
      \leq \frac{\kappa_\alpha}{\mathbf{c}^{\alpha+|\gamma|}}\cdot
      \frac{|z(y)|}{|x_0-y|^{\alpha+|\gamma|}}
      =: 
     \Theta_{\gamma,x_0,r}(y), 
     \end{align*}
     for
     every
     $x\in B(x_0,r)$ and every $y\in \RN\setminus\mathcal{O}_\delta$, and 
     this gives~\eqref{eq.touseDerivativebound}.
     Moreover, since $x_0\in\mathcal{O}_{\delta/2}$,
     from~(i) we infer that $\Theta_{\gamma,x_0,r}\in L^1(\RN)$, 
     thus showing~\eqref{yfewt8476874679}. This ends the proof.
    \end{proof}    
    \begin{lemma} \label{lem.multiplication}
     Let $u\in H^1(\RN)$ be a solution of~\eqref{eq.mainPDE},
     and let $\OO$ be an open set with
     $\overline{\OO}\subseteq\Omega$.
     Let~$\rho := \mathrm{dist}(\OO,\de\Omega)$ and~$\zeta\in C_0^\infty(\RN,\R)$ satisfy
     \begin{itemize}
      \item[{(i)}] $\zeta\equiv 1$
      on $\OO_{\rho/4}$;
      \item[{(ii)}]
      $\mathrm{supp}(\zeta)\subseteq \OO_{\rho/2}$;
      \item[{(iii)}] $0\leq \zeta\leq 1$ on $\RN$;
     \end{itemize}
     with the notation introduced in~\eqref{05785987vufdhg}.
     
     Then, there exists $\psi\in C^\infty(\overline{\OO},\R)$ such that
     $v:=u\,\zeta$
     is a weak solution of 
     \begin{equation} \label{eq.equationforvuzeta}
      \LL v = f+\psi\qquad\text{in $\OO$}.
      \end{equation}
    \end{lemma}  
    \begin{proof}
     First of all, since $u\in H^1(\RN)$ and
     $\zeta\in C_0^\infty(\RN,\R)$, one has that~$v\in H^1(\RN)$. Moreover, 
     we set~$\omega := u(1-\zeta)$ and we observe that~$\omega=u-u\zeta=u-v$.
     Since~$u$
     solves~\eqref{eq.mainPDE}, from~\eqref{eq.weaksoldef} we deduce that,
     for any $\varphi\in C_0^\infty(\OO,\R)$,
     \begin{equation}\label{hfurghuerghb} 
     \begin{split}
       \int_\OO f\varphi\,\d x 
     =\;& \int_\OO\langle\nabla u,\nabla\varphi\rangle\,\d x
      + \frac{c_{N,s}}{2}\int_{\RN\times\RN}
     \frac{(u(x)-u(y))(\varphi(x)-\varphi(y))}{|x-y|^{N+2s}}\,\d x\,\d y
     \\= \;&\int_{\OO}\langle \nabla v,\nabla\varphi\rangle\,\d x
      + \frac{c_{N,s}}{2}\int_{\RN\times\RN}
     \frac{(v(x)-v(y))(\varphi(x)-\varphi(y))}{|x-y|^{N+2s}}\,\d x\,\d y
     \\
     & \qquad\qquad
     + \frac{c_{N,s}}{2}\int_{\RN\times\RN}
     \frac{(\omega(x)-\omega(y))(\varphi(x)-\varphi(y))}{|x-y|^{N+2s}}\,\d x\,\d y.
     \end{split}
     \end{equation}
     We now observe that, since $u\in H^1(\RN)$ and $1-\zeta$ is smooth on $\RN$,
     both $\omega$ and $|\omega|$ are in $H^1(\RN)$. 
     Moreover, since $\zeta \equiv 1$ on $\OO_{\rho/4}$, one has
    \begin{equation}\label{y74tthy54u65u65ug}
    \omega = u(1-\zeta)\equiv 0 \qquad\text{on $\OO_{\rho/4}$}.\end{equation}
     Therefore, we are in the position to apply Lemma~\ref{lem.generalII}
     (with $\alpha := N+2s > N$, $\delta:=\rho/4$, and either~$z:=\omega$ or~$z:=|\omega|$),
     and so, recalling~\eqref{eq.smoothII}, we see that
     \begin{align*}
      & \II_{N+2s}[\omega](x) = \int_{\RN}\frac{\omega(y)}{|x-y|^{N+2s}}\,\d y
      \qquad\text{and} \qquad \II_{N+2s}\big[|\omega|\big](x) = 
      \int_{\RN}\frac{|\omega(y)|}{|x-y|^{N+2s}}\,\d y
     \end{align*}
     are (well-posed and) smooth on $\OO_{\rho/8}$.  
     In particular, using this fact, and recalling~\eqref{y74tthy54u65u65ug} we obtain that,
     for any~$\varphi\in C_0^\infty(\OO,\R)$,
        \begin{align*}
      & \int_{\RN\times\RN}
      \frac{|\omega(x)-\omega(y)|\cdot|\varphi(x)-\varphi(y)|}{|x-y|^{N+2s}}\,\d x\,\d y
      \\[0.1cm]
      & \qquad\leq
      2\int_{\RN\times\RN}
      \frac{|\omega(x)-\omega(y)|}{|x-y|^{N+2s}}\cdot|\varphi(x)|\,\d x\,\d y \\[0.1cm]
      & \qquad = 2\int_{\RN}\bigg(\int_{\RN}
      \frac{|\omega(x)-\omega(y)|}{|x-y|^{N+2s}}\,\d y\bigg)|\varphi(x)|\,\d x
      \\[0.1cm]
      & \qquad =
      2\int_{\OO}\bigg(\int_{\RN}
      \frac{|\omega(y)|}{|x-y|^{N+2s}}\,\d y\bigg)|\varphi(x)|\,\d x\\[0.2cm]
      & \qquad
      = 2\int_{\OO}\II_{N+2s}\big[|\omega|\big](x)\,|\varphi(x)|\,\d x
      \\[0.2cm]
      & \qquad
      \leq 2 {|\mathcal{O}|}\max_{\overline{\OO}}\big(\II_{N+2s}\big[|\omega|\big]\big)\cdot
      \max_{\RN}|\varphi| < \infty
     ,\end{align*}
where~$|\cdot|$ denotes the standard Lebesgue measure in $\RN$.
	 Thanks to the above estimate we can apply
	 Fubini's theorem, thus giving
	 \begin{align*}
	  & 
	  \int_{\RN\times\RN}
     \frac{(\omega(x)-\omega(y))(\varphi(x)-\varphi(y))}{|x-y|^{N+2s}}\,\d x\,\d y
     \\[0.1cm]
     & \qquad = 2\int_{\RN\times\RN}
     \frac{\omega(x)-\omega(y)}{|x-y|^{N+2s}}\cdot\varphi(x)\,\d x\,\d y
     = 2\int_{\OO}\II_{N+2s}[\omega](x)\,\varphi(x)\,\d x.
	 \end{align*}	  
Plugging this information into~\eqref{hfurghuerghb}, we get
	 \begin{align*}
	 &  \int_\OO f\varphi\,\d x 
	 = 
	 \int_{\OO}\langle \nabla v,\nabla\varphi\rangle\,\d x
      + \frac{c_{N,s}}{2}\int_{\RN\times\RN}
     \frac{(v(x)-v(y))(\varphi(x)-\varphi(y))}{|x-y|^{N+2s}}\,\d x\,\d y \\[0.1cm]
     & \qquad\qquad\qquad+ c_{N,s}\int_{\OO}\II_{N+2s}[\omega](x)\,\varphi(x)\,\d x.
	 \end{align*}
	In its turn, this identity gives
	\begin{align*}
	& \int_{\OO}\langle \nabla v,\nabla\varphi\rangle\,\d x
      + \frac{c_{N,s}}{2}\int_{\RN\times\RN}
     \frac{(v(x)-v(y))(\varphi(x)-\varphi(y))}{|x-y|^{N+2s}}\,\d x\,\d y 
     \\[0.1cm]
     & \qquad
     = \int_{\OO}\big(f-c_{N,s}\II_{N+2s}[\omega]\big)\varphi\,\d x
     \qquad \text{for all $\varphi\in C_0^\infty(\OO,\R)$},
     \end{align*}
	 which shows that $v$ satisfies~\eqref{eq.equationforvuzeta}
	 (with $\psi := -c_{N,s}\II_{N+2s}[\omega]$).
	 Finally, since
	 we know that $\II_{N+2s}[\omega]$
	 is smooth on $\OO_{\rho/8}$, we conclude that
	 also $\psi$ is smooth
	 on~$\overline{\OO}$), and the proof is complete.
    \end{proof}
    
    With Lemma~\ref{lem.multiplication}, we can finally prove
    Theorem~\ref{thm.higherregul}.
    
    \begin{proof}[Proof of Theorem~\ref{thm.higherregul}]
     As in the proof of Proposition~\ref{prop.higherregweak},
     we proceed by induction on $m\in\N\cup\{0\}$.
     First of all, the case $m = 0$ (that is, $f\in L^2(\Omega)$)
     is nothing but Theorem~\ref{thm.mainregul}; we then assume that the desired result
     holds for a certain integer $m\geq 0$, and we prove that
     it still holds for $m +1$. \medskip
     
     Let~$V$ be a fixed open set satisfying
     $\overline{V}\subseteq\Omega$.
     Moreover, let $f\in H^{m+1}(\Omega)$ and let 
     $u\in H^{1}(\RN)$ be a solution of~\eqref{eq.mainPDE}. Since, in particular, $f\in H^m(\Omega)$,
     from the inductive hypothesis we derive that
     $u\in H^{m+2}_\loc(\Omega)$, and thus
     \begin{equation} \label{eq.uHregloc} 
     u\in H^{m+2}(V).
     \end{equation}
     We now arbitrarily fix an open set $\OO\subseteq\RN$ such that
     $\overline{V}\subseteq\OO$ and $\overline{\OO}\subseteq\Omega$,
     and we set 
     $$\rho := \mathrm{dist}(\OO,\de\Omega) > 0.$$
     Moreover,
     we choose a cut-off function $\zeta\in C_0^\infty(\RN,\R)$
     satisfying (i)-(ii)-(iii) in Lemma~\ref{lem.multiplication}, and we define
     $v := u\,\zeta$.   
     On account of~\eqref{eq.uHregloc}, and since
     $\mathrm{supp}(\zeta)\subseteq\Omega$, we have
     $$v\in H^{m+2}(\RN).$$ 
     Moreover, from Lemma~\ref{lem.multiplication} we know that
     there exists a function
     $\psi\in C_0^\infty(\overline{\OO},\R)$ such that
     $v$ is a weak solution of the equation
     $$\LL v = f + \psi\qquad\text{in $\OO$}.$$
     Since $f\in H^{m+1}(\Omega)$ and $\psi$ is smooth on an open
     neighborhood of $\overline{\OO}$, we derive that
     $g := f+\psi\in H^{m+1}(\OO)$. As a consequence, we can apply
     Proposition~\ref{prop.higherregweak}, thus obtaining that
     $$\text{$v\in H^{m+3}_\loc(\OO)$ and, in particular, $v\in H^{m+3}(V)$}.$$
     {F}rom this, since $\zeta\equiv 1$ on $\OO_{\rho/4}\supset V$,
     we conclude that $$v\equiv u \in H^{m+3}(V),$$ 
     and the proof is finally complete.
    \end{proof}    
    By combining Theorem~\ref{thm.higherregul} with
    the well-known Sobolev Embedding theorems, we immediately 
    obtain the Corollary~\ref{cor.regulSmoot} below.    
    \begin{corollary} \label{cor.regulSmoot}
    Let $m\in\mathbb{N}$ satisfy $m > N/2$ and let 
    $f\in C^m(\Omega,\R)$. 
    Moreover, let $u\in H^1(\RN)$ a weak solution of~\eqref{eq.mainPDE}.
    Then, there exist
	 a non-negative integer $k = k_{m,N}$ and a \emph{(}unique\emph{)}
    function $\widehat{u}\in C^k(\Omega,\R)$ such that
    $$\widehat{u}\equiv u\qquad\text{a.e.\,on $\Omega$}.$$
    More precisely, the number $k$ is explicitly given by
    \begin{equation} \label{eq.defkmN}
     k = k_{m,N} = \begin{cases}
     \left[m-\frac{N}2\right], & \text{if $m-\frac{N}2\notin\N$}, \\[0.15cm]
     m-\frac{N}2-1 & \text{if $m-\frac{N}2\in\N$}.
    \end{cases}
    \end{equation}
    In particular, if $f\in C^\infty(\Omega,\R)$, then $\widehat{u}\in C^\infty(\Omega,\R)$.
    \end{corollary}  
    \subsection{Boundary regularity and proofs of Theorems~\ref{BDTH} and~\ref{thm.RegulBoundaryII}} 
    \label{subsec.boundary}
    Now that we have e\-sta\-bli\-shed interior regularity
    for the weak solutions of~\eqref{eq.mainPDE}, we focus on the
    \emph{boundary} regularity for the weak solutions of~\eqref{eq.mainPB}. 
    
  	To begin with, we prove the following theorem.
  	
    \begin{theorem}\label{BOUNDED} 
     Assume that\footnote{We observe that when $N=1,2$ 
     the boundedness of the energy solutions follows directly by Sobolev Embedding.} 
	$N\geq 3$, and
	let $f\in L^{p}(\Omega)$, with $p>N/2$. Moreover, 
 	assume that there exists the weak solution ${u}_f\in \spX$
	of~\eqref{eq.mainPB}. 
	
	Then, $u_f\in L^\infty(\RN)$ and
	\begin{equation}\label{3HY-0} 
	\|{u}_f\|_{L^\infty(\RN)}\le C\,\|f\|_{L^{p}(\Omega)},
	\end{equation}
	for some constant $C > 0$.
   \end{theorem}
   
   \begin{remark} \label{rem.existWeakBounded}
    We point out that, on account of Theorem~\ref{thm.existenceLax},
    a (unique) weak solution of~\eqref{eq.mainPB} exists
    if $f\in L^p(\Omega)$ with $p\geq 2$.
   \end{remark}
   
	\begin{proof}[Proof of Theorem~\ref{BOUNDED}] The proof employs the classical method by Stampacchia,
	as extended in the nonlocal setting, see e.g. the proof of Proposition~9
	in~\cite{SV} and of Theorem 2.3 in~\cite{IRE}. 
	We give full details for the reader's convenience. \medskip
	
	Let $\delta>0$ to be conveniently chosen later on.
	Assuming that $u_f$ is not identically zero
	(otherwise there is nothing to prove),
    we set
	\begin{equation}\label{uASC} 
	\tilde u:=\frac{\delta\,{u}_f}
	{\|{u}_f\|_{L^{2^*}(\Omega)}+ \|f\|_{L^{p}(\Omega)}}
	\qquad{\mbox{and}}
	\qquad \tilde f:=\frac{\delta\,f}{\|{u}_f\|_{L^{2^*}(\Omega)}+\|f\|_{L^{p}(\Omega)}},
	\end{equation}
where
	$ 2^* :=\frac{2N}{N-2}$.
	In this way, we have that
	\begin{equation}\label{EQHS}
	\begin{cases}
{\mathcal{L}} \tilde u=\tilde f & \text{in $\Omega$}, \\
	\tilde u=0 &  \text{in $\RN\setminus\Omega$}.
	\end{cases}
	\end{equation}
	Also, for every $k\in\N$, we define $C_k := 1-2^{-k}$ and
	$$v_k:=\tilde u-C_k, \quad w_k:=(v_k)_+:=\max\{v_k,0\},\quad
	U_k:= \|w_k\|_{L^{2^*}(\Omega)}^2.$$ 
	We point out that, by the Dominated
	Convergence Theorem,
	\begin{equation}\label{IJS:odf} 
	\lim_{k\to+\infty}U_k=
	\lim_{k\to+\infty}\|w_k\|^2_{L^{2^*}(\Omega)}=\|(\tilde u-1)_+\|^2_{L^{2^*}(\Omega)}.
	\end{equation}
	Also, if we take $k:=0$, we see that 
	$w_0=(v_0)_+ = (\tilde u-C_0)_+=\tilde u_+$, and thus
	\begin{equation}\label{DHAI:1}
	U_0=\left( \int_\Omega
	w_0^{2^*}(x)\,\d x\right)^{2/2^*}
	\leq 
	\left( \int_\Omega \tilde u^{2^*}(x)\,\d x\right)^{2/2^*}
	= \|\tilde u\|^2_{L^{2^*}(\Omega)}\le\delta^2,
	\end{equation}
	which can be taken conveniently small in what follows.
	In addition, in $\RN\setminus\Omega$ we have that 
	$v_{k+1}=-C_{k+1} \leq 0$ and thus
	$$w_{k+1}=0.$$
	We can then use $w_{k+1}$ as test function
	and deduce from~\eqref{EQHS} that
	\begin{equation}\label{EAFS1}
	\begin{split}
	& \int_\Omega\nabla w_{k+1}(x)\cdot\nabla\tilde u(x)\,\d x+
	\iint_{\RN\times\RN}\!\!\!\!\!
	\frac{(w_{k+1}(x)-w_{k+1}(y))(\tilde u(x)-\tilde u(y))}
	{|x-y|^{N+2s}}\,\d x\,\d y\\
	& \qquad = \int_\Omega w_{k+1}(x)\,\tilde f(x)\,\d x.
	\end{split}
	\end{equation}
	We also remark that, for a.e.\,$x,y\in\RN$, we have (see, e.g.,~\cite[Lemma 10]{SV})
	\begin{equation}\label{EAFS2}
	 \begin{split} |w_{k+1}(x)-w_{k+1}(y)|^2
	 & = |(v_{k+1})_+(x)-(v_{k+1})_+(y)|^2 \\[0.2cm] 
	 & \leq ( (v_{k+1})_+(x)-(v_{k+1})_+(y))(v_{k+1}(x)-v_{k+1}(y)) \\[0.2cm]
	&= (w_{k+1}(x)-w_{k+1}(y))(\tilde u (x)-\tilde u(y)).
	\end{split}
	\end{equation}
 	Moreover,
	\begin{align*}
 	\int_\Omega\nabla w_{k+1}(x)\cdot\nabla\tilde u(x)\,\d x =
	\int_{\Omega\cap\{ \tilde u>C_k\}}
	\nabla v_{k+1}(x)\cdot\nabla\tilde u(x)\,\d x =
	\int_\Omega|\nabla w_{k+1}(x)|^2\,\d x.
	\end{align*}
	{F}rom this,~\eqref{EAFS1} and~\eqref{EAFS2} we conclude that
	\begin{equation*}
	\begin{split}&
	\int_\Omega|\nabla w_{k+1}(x)|^2\,\d x
	\leq
	\int_\Omega w_{k+1}(x)\,\tilde f(x)\,\d x.
	\end{split}
	\end{equation*}
	Hence, by Sobolev Inequality,
	\begin{equation}\label{Thsd-0}
	 U_{k+1} =
	\left(\int_\Omega|w_{k+1}(x)|^{2^*}\,\d x\right)^{2/2^*}\leq 
	C\,\int_\Omega|\nabla w_{k+1}(x)|^2\,\d x  \le C\,
	\int_\Omega w_{k+1}(x)\,|\tilde f(x)|\,\d x,
	\end{equation}
	for some $C>0$. Also, $v_{k+1}\le v_k$ and therefore
	\begin{equation}\label{HA}
	w_{k+1}\leq w_k.
	\end{equation}
	Moreover, we observe that
	$$w_k=(\tilde u-C_k)_+ =\left(\tilde u-C_{k+1}+
	\frac1{2^{k+1}}\right)_+=
	\left(v_{k+1}+\frac1{2^{k+1}}\right)_+,$$
	and, as a result,
	\begin{equation}\label{HHA} 
	\{ w_{k+1}>0\} = \{ v_{k+1}>0\} 
	\subseteq \left\{w_k>\frac1{2^{k+1}}\right\}.
	\end{equation}
	We also observe that
	$$ 2^*-\frac{2^*}{p}-1 > 2^*-\frac{2^*}{N/2}-1
	 =\frac{2N}{N-2}-\frac{4}{N-2}-1=1.$$
	Hence, we can define
	\begin{equation}\label{deq} 
	 q:= 2^*\,\bigg(2^*-\frac{2^*}{p}-1\bigg)^{-1}<2^*.
	 \end{equation}
	We observe that
	$$ q> \frac{2^*}{2^*-1}>1.$$
	In addition,
	\begin{equation}\label{EXAO} 
	 \frac1{2^*}+\frac1p+\frac1q=1.
	\end{equation}
	{F}rom this,~\eqref{HA} and~\eqref{HHA}, using the  H\"older Inequality
	with exponents $2^*$, $p$ and $q$, we deduce that
	\begin{equation}\label{Thsd-1}
	\begin{split}
	 &\int_\Omega w_{k+1}(x)\,|\tilde f(x)|\,\d x=
	\int_{\Omega\cap\{ w_{k+1}>0\}} w_{k+1}(x)\,|\tilde f(x)|\,\d x \\[0.2cm] 
	&\qquad \le \|\tilde f\|_{L^{p}(\Omega)}\,\| w_{k+1}\|_{L^{2^*}(\Omega)}\,
	|\Omega\cap\{ w_{k+1}>0\}|^{1/q}\\[0.2cm] 
	&\qquad\le \|w_{k}\|_{L^{2^*}(\Omega)}\,
	\left|\Omega\cap\left\{w_k>\frac1{2^{k+1}}\right\}\right|^{1/q}\\[0.2cm] 
	&\qquad\le U_k^{1/2}\,\left( 2^{2^*(k+1)}
	\int_{\Omega\cap\left\{w_k>\frac1{2^{k+1}}\right\}} w_k^{2^*}\right)^{1/q} \\[0.2cm] 
	&\qquad\le \tilde C^k\,U_k^{1/2}\,U_k^{2^*/(2q)},
	\end{split}
	\end{equation}
	for some $\tilde C>1$. 
	We now define
	$$\beta:=\frac{1}{2} +\frac{2^*}{2q},$$
	and we stress that
	\begin{equation}
	\beta>1,
	\end{equation}
	thanks to~\eqref{deq}. Using this notation, we deduce from~\eqref{Thsd-0}
	and~\eqref{Thsd-1} that
	$$U_{k+1}\le \hat C^k\;U_k^\beta,$$
	for some $\hat C>1$.
	As a result, recalling~\eqref{DHAI:1} (and supposing $\delta>0$
	appropriately small), we conclude that
	$$ \lim_{k\to+\infty} U_k=0.$$
	This and~\eqref{IJS:odf} give that
	$$\text{$\|(\tilde u-1)_+\|^2_{L^{2^*}(\Omega)}=0$},$$
	and therefore $\tilde u\leq 1$. As a consequence, recalling~\eqref{uASC}, for every $x\in\Omega$,
	\begin{equation}\label{3HY} 
	 {u}_f(x)\le \frac{\|{u}_f\|_{L^{2^*}(\Omega)}+\|f \|_{L^{p}(\Omega)}}{\delta}.
	\end{equation}
	On the other hand, by testing the
equation against $u_f|_{\Omega}$ 
	(see Remark~\ref{rem.WeakSoluf}),
	and recalling once again relation~\eqref{EXAO},
	we see that
	\begin{align*}
	\int_\Omega |\nabla  u_f(x)|^2\,\d x &\leq
	\int_\Omega |\nabla  u_f(x)|^2\,\d x+
	\iint_{\RN\times\RN}\frac{(u_f(x)-u_f(y))^2}{|x-y|^{N+2s}}\,\d x\,\d y\\[0.2cm]
	& = \int_\Omega u_f(x)\,f(x)\,\d x \\[0.2cm] 
	&\le  
	\|u_f\|_{L^{2^*}(\Omega)}\,\| f\|_{L^{p}(\Omega)}\,|\Omega|^{1/q}.
	\end{align*}
	This and the Sobolev Inequality give that
	$$\|u_f\|_{L^{2^*}(\Omega)}\le \bar{C}\,\| f\|_{L^{p}(\Omega)},$$ 
	for a suitable $\bar{C}>0$. Combining this with~\eqref{3HY} 
	we obtain~\eqref{3HY-0}, as desired.
	\end{proof}
	
	With Theorem~\ref{BOUNDED} at hand, we now focus on the proof of Theorem~\ref{BDTH}.
   We remark that by saying that a function $u\in H^1(\RN)$ satisfies~\eqref{eq.PbGeneralWeak} 
    we mean, precisely, that 
  	$u\leq\bar C$ a.e.\,in $\RN$, $u\equiv 0$ a.e.\,in $\RN\setminus\Omega$, and
    $$
    \int_{\Omega}
     \langle\nabla u,\nabla \varphi\rangle\,\d x
     +\frac{c_{N,s}}{2}\int_{\RN\times\RN}
     \!\!\!\!\!\frac{(u(x)-u(y))(\varphi(x)-\varphi(y))}{|x-y|^{N+2s}}\,\d x\,\d y
     \leq \bar{C}\int_{\Omega}\varphi\,\d x,$$
     for every \emph{non-negative} function $\varphi\in C_0^\infty(\Omega,\R)$.

    The proof of Theorem~\ref{BDTH} relies on the construction
	of an appropriate barrier, which will be built by recursive corrections
	of monomial functions. In fact, the arguments provided have
	wider applicability and can be exploited in more general contexts as well,
	but for concreteness we will follow on the specific operator,
	boundary conditions and geometry dealt with
	in this article.  \medskip

The first step towards the proof of Theorem~\ref{BDTH}
consists in an elementary computation on functions
which have a convex portion in their graphs.
\begin{lemma} \label{EDASlema}
 Let $s\in(0,1)$, $d>\ell>0$ and~$v:\R\to\R$. Assume that
 $$v\in C^{1,1}((-\infty,d))\cap L^\infty(\R)$$
 and that $v$ is convex in $(-\infty,d)$. 
 Then, for every $x\in(0,\ell)$,
 $$ (-\Delta)^s v(x)
 \le\frac{2c_{1,s}\,\|v\|_{L^\infty(\R)}}{s\,(d-\ell)^{2s}}.$$
 \end{lemma}
 \begin{proof} 
 In the principal value sense, we have that, for every $x\in(0,\ell)$,
 \begin{align*}
 & \int_{2x-d}^d \frac{v(y)-v(x)}{|x-y|^{1+2s}}\,\d y =
  \int_{x-d}^{d-x} \frac{v(x+z)-v(x)}{|z|^{1+2s}}\,\d z=
  \int_{x-d}^{d-x} \frac{v(x+z)-v(x)-v'(x)z}{|z|^{1+2s}}\,\d z\ge0,
  \end{align*}
 thanks to the convexity assumption. As a result,
 $$ \frac{(-\Delta)^s v(x)}{c_{1,s}}\le \int_{\R\setminus
 (2x-d,d)} \frac{v(x)-v(y)}{|x-y|^{1+2s}}\,\d y
 =\int_{\R\setminus
 (x-d,d-x)} \frac{v(x)-v(x+z)}{|z|^{1+2s}}\,\d z. $$
 We stress that, in the latter integral, we have that $|z|\ge d-x\ge d-\ell$.
 Hence,
 $$ \frac{(-\Delta)^s v(x)}{c_{1,s}}\le 2\|v\|_{L^\infty(\R)}\int_{\{
 |z|\ge d-\ell\}} \frac{\d z}{|z|^{1+2s}}=
 \frac{2\|v\|_{L^\infty(\R)}}{s\,(d-\ell)^{2s}},$$
 and the proof is complete.
\end{proof}
 The next auxiliary result for the proof of Theorem~\ref{BDTH}
 focuses on a calculation for a modified monomial function.
 \begin{lemma}\label{Pa0s}
 Let $s\in(0,1)$, $L>0$ and $\alpha\ge2s$. Let also
 \begin{equation}\label{nawa}
 w_\alpha(x):=\begin{cases}
  x_+^\alpha & \text{if $x<2L$},\\
  (2L)^\alpha & \text{if $x\geq 2L$}.
 \end{cases}
 \end{equation}
 Then, there exists $C>0$, only depending on $L$,
 $s$ and $\alpha$, such that, for all $x\in(0,L)$,
 \begin{equation}\label{EQUA:Z} 
  |(-\Delta)^s w_\alpha(x)|\le 
  \begin{cases}
	C & \text{if $\alpha>2s$},\\
	C\,\big(1+|\log x|\big) & \text{if $\alpha=2s$}.
\end{cases}
\end{equation}
\end{lemma}

 \begin{proof} 
 Up to scaling, it is not restrictive
 to suppose that $L:=1$.
 Given any point $x\in(0,1)$, we use the substitution $z:=y/x$ to see that
 \begin{align*}\frac{
(-\Delta)^s w_\alpha(x)}{c_{1,s}} = 
\int_\R\frac{x^\alpha-\min\{y_+^\alpha,2^\alpha\}}{|x-y|^{1+2s}}\,\d y = x^{\alpha-2s}
 \int_\R\frac{1-\min\{z_+^\alpha,(2/x)^\alpha\}}{|1-z|^{1+2s}}\,\d z,
\end{align*}
 where the principal value notation has been omitted for the sake of
 shortness. Then, we observe that
 \begin{align*}
 A_1  :=\left|\int_{-\infty}^0\frac{1-\min\{z_+^\alpha,(2/x)^\alpha\}}{|1-z|^{1+2s}}\,\d z\right|
 =\int_{-\infty}^0\frac{\d z}{(1-z)^{1+2s}} 
 =\int^{+\infty}_1\frac{\d t}{t^{1+2s}}=\frac{1}{2s}.
 \end{align*}
 Similarly,
 \begin{align*}
 A_2 & :=
 \left|\int_2^{+\infty}\frac{1-\min\{z_+^\alpha,(2/x)^\alpha\}}{|1-z|^{1+2s}}\,\d z\right|
 \le \int_2^{+\infty}\frac{1+\min\{z_+^\alpha,(2/x)^\alpha\}}{(z-1)^{1+2s}}\,\d z\\[0.1cm]
 & =\int_1^{+\infty}\frac{\d t}{t^{1+2s}}+
 \int_2^{2/x}\frac{ z^\alpha}{(z-1)^{1+2s}}\,\d z+\left(\frac2x\right)^\alpha\int_{2/x}^{+\infty}
 \frac{\d z}{(z-1)^{1+2s}}\,\d z\\[0.1cm]
 & \le C_1(1+x^{2s-\alpha}\ell(x)),
 \end{align*}
 for some $C_1>0$, where
 $$\ell(x):=\begin{cases}
  1+|\log x|& \text{if $\alpha=2s$},\\
 1 & \text{otherwise}.
\end{cases}$$
 In addition, using the principal value notation,
 \begin{align*}
 A_3 & :=
 \left|\int_0^2\frac{1- \min\{z_+^\alpha,(2/x)^\alpha\} }{|1-z|^{1+2s}}\,\d z\right|
 =\left|\int_0^2\frac{1-z^\alpha}{|1-z|^{1+2s}}\,\d z\right| \\[0.2cm]
 & =
 \left|\int_{-1}^1\frac{1-(1+t)^\alpha}{|t|^{1+2s}}\,\d t\right|=
 \left|\int_{-1}^1\frac{(1+t)^\alpha-1-\alpha t}{|t|^{1+2s}}\,\d t\right|\le C_2,
 \end{align*}
 for some $C_2>0$. All in all, we find that
 $$\frac{|(-\Delta)^s w_\alpha(x)|}{c_{1,s}}\le
 x^{\alpha-2s}(A_1+A_2+A_3)\le A_1+A_3+x^{\alpha-2s}A_2\le C_3(1+\ell(x)),$$
 for some $C_3>0$, yielding the desired result.
\end{proof}
{F}rom Lemma~\ref{Pa0s}, we obtain the following barrier.
\begin{lemma}\label{l02}
 Let $s\in(0,1)$. There exist 
 a number $d>0$ and a function 
 $\beta\in \CCs(\R)\cap C^2((0,d),\R)$ satisfying the following properties:
\begin{itemize}
\item there exists $C_0>0$ such that for all $x\ge d$
\begin{equation}\label{EQUA:LLPIU}\beta(x)\ge C_0\,;
\end{equation}
\item for all $x\le0$,
\begin{equation}\label{EQUA:LL}\beta(x)=0\,;\end{equation}
\item there exists $C_1\ge1$, independent of $d$, such that 
\begin{equation}\label{EQUA:LR} 
\frac{x}{C_1}\le\beta(x)\le C_1 x \qquad\text{for all $x\in(0,d)$};
\end{equation}
\item there exists $C_2>0$, independent of $d$, such that
\begin{equation}\label{EQUA:L}{\mathcal{L}}\beta(x)\ge-C_2\qquad
\text{for all $x\in(0,d)$.}
\end{equation}
\end{itemize}
Furthermore, $\beta\in H^1_{\loc}(\R)$.
\end{lemma}

\begin{proof} 
 We distinguish two cases, according to the value of $s$. \medskip
 
 \textsc{Case I:} $s\in(1/2,1)$.
 In this case, we let
 $$
 \rho(s) := \frac{2s-1}{2(1-s)}\qquad\text{and}\qquad
 J := \begin{cases}
 [\rho(s)] & \text{if $\rho_s\notin\N$}, \\
 \rho(s)-1 & \text{otherwise}.
 \end{cases}
 $$
 Also, for each $j\in\N$ with $0\le j\le J+1$, we set $\alpha_j:=1+2j(1-s)$.
 We observe that, for all $j\in\{0,\dots,J\}$, we have
 \begin{equation*}
 \alpha_j\le 1+2J(1-s)<1+(2s-1)=2s.
 \end{equation*}
 Therefore, for all $j\in\{0,\dots,J\}$,
 we can define $(-\Delta)^s x_+^{\alpha_j}$ and,
 by homogeneity, we see that, for all $x>0$,
 $$ (-\Delta)^s x_+^{\alpha_j}=\kappa_j\,x_+^{\alpha_j-2s},$$
 for a suitable $\kappa_j\in\R$.
 As a matter of fact, since $\alpha_j\ge1$, we have that $x_+^{\alpha_j}$
 is a convex function and therefore $ (-\Delta)^s x_+^{\alpha_j}<0$
 in $(0,+\infty)$. {F}rom this, we get
 \begin{equation}\label{PSxdm}
 \kappa_j<0\qquad\text{for every $j = 0,\ldots,J$}.
 \end{equation}
 We also point out that, for every $j\in\{0,\dots,J \}$,
 \begin{align*} 
  \alpha_j-2s & =\alpha_j-2+2(1-s)= 1+2j(1-s)-2+2(1-s) \\
  & =1+2(j+1)(1-s)-2 = \alpha_{j+1}-2.
 \end{align*}
 Now we define $\{c_0,\dots,c_{J+1}\}$ as follows. We let $c_0:=1$, and then,
 recursively,
 for every index $j\in\{1,\dots,J+1\}$,
 \begin{equation}\label{0.7bis} 
  c_j:=-\frac{\kappa_{j-1} \,c_{j-1}}{\alpha_j\,(\alpha_j-1)}.
 \end{equation}
 We stress that this definition is well posed, since, if $j\in\{1,\dots,J\}$,
 $$\alpha_j\ge 1+2(1-s)>1.$$
 {F}rom this and~\eqref{PSxdm}, it follows that
 \begin{equation}\label{7rdfuwyhvzcv}
  c_j>0\qquad\text{for every $j = 0,\ldots,J+1$}.
  \end{equation}
 Hence, we consider the function $\tilde\beta:\R\to\R$ defined as follows
\begin{equation}\label{HIGHE}\tilde\beta(x):=\sum_{j=0}^J c_j\,x^{\alpha_j}_+.\end{equation}
 Since $\alpha_j < 2s$ for every $0\leq j\leq J$, 
 it is easy to recognize that
 \begin{itemize}
  \item $\tilde\beta(x)\ge0$ for all $x\in\R$;
  \item $\tilde\beta\in \CCs(\R)\cap C^2((0,\infty),\R)$
 and $\tilde{\beta}\in H^1_{\loc}(\R)$.
 \end{itemize}
Moreover, for every $x>0$ we have
 \begin{eqnarray*}
  {\mathcal{L}}\tilde\beta(x)& =& 
-\sum_{j=0}^J c_j\alpha_j(\alpha_j-1)\,x^{\alpha_j-2}+
 \sum_{j=0}^J c_j\kappa_j\,x^{\alpha_j-2s} \\
 &=& -\sum_{j=1}^J c_j\alpha_j(\alpha_j-1)\,x^{\alpha_j-2}+
 \sum_{j=1}^{J} c_{j-1}\kappa_{j-1}\,x^{\alpha_j-2}
 +c_J\kappa_J\,x^{\alpha_{J}-2s}
 \\
 &= & -2\sum_{j=1}^{J}c_j \alpha_j (\alpha_j -1) \,x^{\alpha_j-2}+ \,  c_J\kappa_J\,x^{\alpha_{J}-2s},
 \end{eqnarray*}
 where~\eqref{0.7bis} was used in the latter line.
 As a consequence, taking $d\in(0,1)$ to be chosen conveniently small
 in what follows,
 employing the notation in~\eqref{nawa} with $L:=1$,
 and introducing the function
 \begin{equation} \label{eq.defbetasharp}
  \beta_\sharp:=\tilde\beta+c_{J+1} w_{\alpha_{J+1}},
  \end{equation}
 we obtain that $\beta_\sharp$ satisfies the following properties:
 \begin{itemize}
  \item$\beta_\sharp\geq 0$ on $\R$ (as $c_{J+1} > 0$ and $w_{\alpha_{J+1}}\geq 0$ on $\R$);
  \item $\beta_\sharp\in\CCs(\R)\cap C^2((0,d),\R)$ and $\beta_\sharp
  \in H^1_{\loc}(\R)$.
 \end{itemize}
 Furthermore, if $x\in(0,d)$, we get
 \begin{equation*}
 \begin{split}
 {\mathcal{L}} \beta_\sharp(x)
   =\;& -2\sum_{j=1}^{J}c_j \alpha_j (\alpha_j -1)
 \,x^{\alpha_j-2}+ \, c_J\kappa_J\,x^{\alpha_{J}-2s}
  -c_{J+1}\Delta w_{\alpha_{J+1}}(x)+c_{J+1}(-\Delta)^sw_{\alpha_{J+1}}(x) \\[0.1cm]
 =\;& -2\sum_{j=1}^{J}c_j \alpha_j (\alpha_j -1) \,x^{\alpha_j-2}+ \, c_J\kappa_J\,x^{\alpha_{J}-2s}
  -c_{J+1} \alpha_{J+1} (\alpha_{J+1}-1) x^{{\alpha_{J+1}-2}}+
  c_{J+1}(-\Delta)^sw_{\alpha_{J+1}}(x) \\
  =\;& -2\sum_{j=1}^{J+1}c_j \alpha_j (\alpha_j -1) \,x^{\alpha_j-2}+ \,c_{J+1}(-\Delta)^sw_{\alpha_{J+1}}(x),
 \end{split}
 \end{equation*}
 where~\eqref{0.7bis} was used once again. {F}rom this
 and~\eqref{EQUA:Z}, we obtain that,
 \begin{equation}\label{EQUA:ZZ}
 {\mathcal{L}}\beta_\sharp(x)\ge -C_\sharp|\log x| 
-2\sum_{j=1}^{J+1}c_j \alpha_j (\alpha_j -1) \,x^{\alpha_j-2}
= C_\sharp\log x -2\sum_{j=1}^{J+1}c_j \alpha_j (\alpha_j -1)
 \,x^{\alpha_j-2},
  \end{equation}
 for all $x\in(0,d)$ and for some $C_\sharp>0$.
 Now, we let
 $$ \tilde W(x):=\frac{ x_+^2}{4}(3-2\log x)_+ + \frac{2}{C_{\sharp}}\sum_{j=1}^{J+1}c_j  \,x_{+}^{\alpha_j}
 \qquad{\mbox{and}}\qquad S(d):=\max_{(-\infty,d]}\tilde W.$$
 Notice that
 $$ \lim_{d\searrow0}\frac{ S(d)}{d}=0.$$
 As a result, by possibly shrinking $d\in(0,1)$, we can
 additionally suppose that
 \begin{equation}\label{CVabsopf}
 S(d)\le \frac{d}{4C_\sharp}.
 \end{equation}
 Then, we take a continuous function
 \begin{equation}\label{0.10bis}
 W:\R\to[0,2S(d)]
 \end{equation} 
 satisfying the following properties:
\begin{itemize}
 \item[(i)] $W(x)=\tilde W(x)$ for all $x\le d$;
 \item[(ii)] $W(x)=0$ for all $x\ge 2d$;  
 \item[(iii)] $W\in C^\infty\left(\left(0,+\infty\right),\R\right)$.
\end{itemize}
 We define
 $$ \beta(x):=\beta_\sharp(x)-C_\sharp W(x).$$
 Notice that, by the regularity of $\beta_\sharp$ and
 $W$,
 we have that
 $$\text{$\beta\in \CCs(\R)\cap C^2((0,d),\R)$ and $\beta\in H^1_{\loc}(\R)$}.$$
 Moreover, if $d>0$ is sufficiently small,
 \begin{equation}\label{0.10ter}
  W(x)=\frac{ x^2}{4}(3-2\log x)  
  +\frac{2}{C_{\sharp}}\sum_{j=1}^{J+1}c_j  \,x^{\alpha_j}
  \qquad {{\mbox{for all $x\in(0,d)$,}} }
  \end{equation} 
  and therefore $$W''(x)=-\log x  
  + \frac{2}{C_{\sharp}}
  \sum_{j=1}^{J+1}c_j  \alpha_j (\alpha_j -1) \,x^{\alpha_j -2}.$$ 
  In addition, for small $d>0$, we have that
  $$\text{$W\in C^{1,a}(-1,2d)$ for some $a\in(0,1)$},$$ 
 and thus $(-\Delta)^s W\in C^{{0,1+a-2s}}(-1,2d)$.
 As a consequence (see, e.g.,~\cite[Propositions 2.1.7 and 2.1.8]{SILV}), 
 we obtain that, for all $x\in(0,d)$,
 $$ |(-\Delta)^s W(x)|\le \tilde C,$$
 for a suitable $\tilde C>0$ only depending on $N,s$ and $a$ (hence,
 in particular, $\tilde C$ is independent of $d$).
 Then, we deduce from~\eqref{EQUA:ZZ} that, for every $x\in(0,d)$,
 \begin{align*}
{\mathcal{L}}\beta(x) &
  \geq 
  C_\sharp\log x - 2\sum_{j=1}^{J+1}c_j  \alpha_j (\alpha_j -1) \,x^{\alpha_j -2} -C_\sharp{\mathcal{L}}W(x) \\
  & = C_\sharp\log x - 2\sum_{j=1}^{J+1}c_j  \alpha_j (\alpha_j -1) \,x^{\alpha_j -2} -C_\sharp\big(\log x+(-\Delta)^sW(x)\big) + 2\sum_{j=1}^{J+1}c_j  \alpha_j (\alpha_j -1) \,x^{\alpha_j -2}\\
   &= -C_\sharp(-\Delta)^sW(x) \\
  &\geq- C_\sharp\,\tilde C.
 \end{align*}
 This proves~\eqref{EQUA:L}.
 Also,~\eqref{EQUA:LL} is obvious. To proceed
 further we observe that, 
 for every $x\in\R$, we have the estimate
\begin{equation}\label{IKSdjddertdd}
 \sum_{j=0}^J c_j\,x^{\alpha_j}_+ +c_{J+1} w_{\alpha_{J+1}}(x)
 \geq c_0\,x^{\alpha_0}_+ = x_+\geq  \min\{x_+,d\}.
 \end{equation}
 As a result, since $\alpha_0=1<\alpha_1\le\alpha_j$ 
 for all $j\in\{2,\ldots,J+1\}$, we
 have that, for every $x\in\R$,
 \begin{equation}\label{0.11bis} 
 \begin{split}
 & \beta(x) + C_\sharp W(x) \\[0.1cm]
 & \qquad =
 \sum_{j=0}^J c_j\,x^{\alpha_j}_+ +c_{J+1} w_{\alpha_{J+1}}
 \in \Big[\min\{x_+,d\},\;\bar C\, \max\{x_+,x^{\alpha_{J+1}}_+\}\Big],
 \end{split}
 \end{equation}
 for a suitable $\bar C>0$. {F}rom this, we conclude that
 \begin{equation}\label{8yiwugdh2ey8whghdfhfhf}
 \beta(x)\ge \min\{x_+,d\} - C_\sharp W(x).
 \end{equation}
 Now, if $x\ge d$, we obtain from~\eqref{CVabsopf},~\eqref{0.10bis}
 and~\eqref{8yiwugdh2ey8whghdfhfhf}
 that
 \begin{equation}\label{8yiwugdh2ey8whghdfhfhf-2}
 \beta(x)\ge d - C_\sharp W(x)\ge
 d - 2C_\sharp \,S(d)\ge\frac{d}2.
 \end{equation}
 If instead $x\in(-\infty,d)$, we deduce from~\eqref{8yiwugdh2ey8whghdfhfhf}
 that
 \begin{equation}\label{8yiwugdh2ey8whghdfhfhf-3} 
 \beta(x)\ge x_+ - \frac{C_\sharp\, x_+^2}{4}(3-2\log x)_+ -2
\sum_{j=1}^{J+1}c_j  \,x_{+}^{\alpha_j} \ge\frac{x_+}2,
 \end{equation}
 by possibly redefining $d>0$ in a conveniently small way.

 We notice that~\eqref{8yiwugdh2ey8whghdfhfhf-2} implies~\eqref{EQUA:LLPIU},
 as desired. In addition,~\eqref{8yiwugdh2ey8whghdfhfhf-3} proves
 the first inequality in~\eqref{EQUA:LR}.
 Besides, from~\eqref{0.10ter}
 and~\eqref{0.11bis} we obtain the second inequality in~\eqref{EQUA:LR},
 thus concluding the proof in the case $s>1/2$. \medskip
 
\textsc{Case II:} $s\in(0,1/2]$. This case is indeed simpler: 
 it suffices
 to rerun the preceding argument starting from~\eqref{eq.defbetasharp},
  taking $\tilde\beta \equiv 0$ and
 $$\beta_\sharp := w_{1} = 
 \begin{cases}
  x_+ & \text{if $x<2$},\\
  2 & \text{if $x\geq 2$}.
 \end{cases}$$ 
 We omit any further detail.
\end{proof}

\begin{corollary}\label{Y:3}
 Let $s\in(0,1)$ be fixed. There exist a number $\ell > 0$
 and 
 a non-negative function $\gamma\in \CCs(\R)\cap C^2((0,\ell),\R)$ such that
 \begin{itemize}
 \item for all $x\le0$,
 \begin{equation}\label{EQUA:LL-1}
  \gamma(x)=0;
  \end{equation}
 \item there exists $c\in(0,1)$ such that for all $x\in(0,\ell)$
 \begin{equation}\label{EQUA:LR-1} 
  cx\le \gamma(x)\le\frac{x}c\;;
  \end{equation}
 \item for all $x\ge\ell$,
 \begin{equation}\label{ZwdZsdZEQUA:L-1}
 \gamma(x)\ge1;
 \end{equation}
 \item for all $x\in(0,\ell)$,
 \begin{equation}\label{EQUA:L-1}
 {\mathcal{L}}\gamma(x)\ge1.
 \end{equation}
 \end{itemize}
 Furthermore, $\gamma\in H^1_{\loc}(\R)$.
\end{corollary}

 \begin{proof} 
 We let $\beta$ as in Lemma~\ref{l02}
 and, in the notation of Lemma~\ref{l02},
 we will choose $\ell\in(0,d/2)$ to be suitably small.
 Given $M>0$, we define
 $$\beta_*(x):=\begin{cases}
 0 & \text{if $x\le0$}, \\
 C_2x^2 & \text{if $x\in(0,\ell)$}, \\
 2C_2\ell x-C_2\ell^2 & \text{if $x\in[\ell,d]$}, \\
 C_2\ell (2 d-\ell)& \text{if $x\in(d,+\infty)$},
 \end{cases} \qquad{\mbox{and}}\qquad
 \gamma(x):= M\big(\beta(x)- \beta_*(x)\big),$$
 where $C_2>0$ is as in Lemma~\ref{l02}.
 We point out that, by taking into account the regularity of $\beta$
 and $\beta_*$, it is very easy to see that
 $$\text{$\gamma\in \CCs(\R)\cap C^2((0,\ell),\R)$ and
 $\gamma\in H^1_\loc(\R)$}.$$
 Moreover, since $\beta_*$ is convex
 in $(-\infty,d)$, by Lemma~\ref{EDASlema}
 there exists a suitable constant $C_3>0$ such that,
 for every $x\in(0,\ell)$, we have
 $$(-\Delta)^s\beta_*\le \frac{ C_3\ell (2 d-\ell)}{(d-\ell)^{2s}}.$$
 As a consequence, as long as $d$ and $\ell$ are sufficiently small we have
 $$ (-\Delta)^s\beta_*\le \frac{ C_3\ell (2 d-\ell)}{(d-\ell)^{2s}}
 \le\frac{ 2 C_3\ell d}{(d/2)^{2s}}
 \le2^{2s+1} C_3\ell d^{1-2s}\le2^{2s} C_3 d^{2-2s}\leq\frac{C_2}2,$$
 Therefore, for all $x\in(0,\ell)$,
 \begin{align*}
   \frac1M\,{\mathcal{L}}\gamma(x)
   \ge {\mathcal{L}}\beta(x)+\Delta\beta_*(x)-\frac{C_2}{2}\geq -C_2+2C_2-\frac{C_2}{2}=\frac{C_2}{2},
 \end{align*}
 thanks to~\eqref{EQUA:L}. By choosing $M\ge\frac{2}{C_2}$, we obtain~\eqref{EQUA:L-1},
 as desired.
 Moreover,~\eqref{EQUA:LL-1} follows from~\eqref{EQUA:LL}. In addition,
 by~\eqref{EQUA:LR}, and taking 
 $$\ell\le\frac{1}{2C_1C_2},$$
 if $x\in(0,\ell)$ we obtain
 $$ \gamma(x)\ge M\left(\frac{x}{C_1}-C_2x^2\right)\ge M\left(\frac1{C_1}-C_2\ell\right)
 x\ge\frac{M x}{2C_1 }.$$ 
 Similarly, recalling~\eqref{EQUA:LR}, we have
 $$\gamma(x)\le MC_1 x.$$
 These considerations imply~\eqref{EQUA:LR-1}.
 Furthermore, exploiting~\eqref{EQUA:LLPIU} and~\eqref{EQUA:LR}, 
 we see that, for every~$x\geq\ell$, one has
 \begin{equation}\label{beta-1}
 \beta(x)\ge\min\left\{ C_0,\frac{\ell}{C_1}\right\}=\frac{\ell}{C_1},
 \end{equation}
 as long as $\ell$ is sufficiently small.
 Moreover, if $x\in[\ell,d]$,
 \begin{equation}\label{beta-2}
 \beta_*(x)\le 2C_2\ell x\le2C_2\ell d.
 \end{equation}
 Similarly, if $x>d$,
 $$ \beta_*(x)\le 2C_2\ell d.$$
 This and~\eqref{beta-2} give that, for all $x\ge \ell$,
 $$ \beta_*(x)\le 2C_2\ell d\le \frac{\ell}{2C_1},$$
 provided that $d$ is chosen sufficiently small.
 {F}rom this and~\eqref{beta-1}, we get
 $$ \gamma(x)\ge M\left(\frac{\ell}{C_1}-\frac{\ell}{2C_1}\right)=\frac{M\ell}{2C_1} 
 \ge1,$$
 as long as $M\ge\frac{2C_1}\ell$.
 This gives~\eqref{ZwdZsdZEQUA:L-1}, as desired.
\end{proof}

 The function $\gamma$ constructed in Corollary~\ref{Y:3} would provide a \alt96 good' barrier for
 the proof of Theorem~\ref{BDTH} 
 if, in addition, $\gamma\in H^1(\R)$. In fact,
 since we aim to apply the weak maximum principle 
 in Theorem~\ref{thm.WMPweak} to the function
 $$\pm u - M\gamma$$
 (for a suitable $M\in\R$), and since $u\in H^1(\RN)$, it is crucial
 to have 
 \begin{equation}\label{hope}
  \gamma\in H^1(\R).
  \end{equation} 
  On the other hand, since
 property~\eqref{ZwdZsdZEQUA:L-1} shows that~\eqref{hope} cannot hold, 
 we need to perform a suitable
 truncation argument. 
 This is described in the next lemma.
 
 \begin{lemma} \label{lem.cutOffpergamma}
 Let $\OO\subseteq\RN$ be a \emph{bounded} 
 open set, and let 
 $\rho > 0$ be such that $\OO\subseteq B(0,\rho)$. Moreover,
 let $f\in\mathcal{C}_s(\RN)\cap C^2(\OO,\R)$.
 Finally, let $R > 4\rho$ and let $\varphi = \varphi_R\in C^\infty_0(\RN,\R)$ satisfy the properties
 \begin{itemize}
  \item[{(i)}] $\varphi\equiv 1$ on $B(0,R)$;
  \item[{(ii)}] $0\leq \varphi\leq 1$ on $\RN$;
  \item[{(iii)}] $\mathrm{supp}(\varphi)\subseteq B(0,2R)$.
 \end{itemize}
 Then, setting $f^\star:=f\varphi$, for all $x\in \OO$ we have
 $$\big|{\mathcal{L}}f^\star(x)-
{\mathcal{L}}f(x)
\big|\le C\,\left( \frac{|f(x)|}{R^{2s}}+\kappa(R)\right),$$
 where $C>0$ is a constant only depending on $N, s$ and $\rho$, and
 $$ \kappa(R):=\int_{\{|y|\geq R\}}\frac{|f(y)|}{1+|y|^{N+2s}}\,\d y.$$
\end{lemma}
\begin{proof} 
 We first observe that, since $R > 4\rho$, we have
 $$R- {\rho}\ge\frac{R}2+\rho\ge\frac{R}2.$$
 As a consequence, if $x\in \OO\subseteq B(0,\rho)$ and $y\in \R^N\setminus B(0,R)$, we have
 $$|x-y|\ge|y|-|x|\ge R-\rho\ge \frac{R}2.$$
 Moreover, one also has
 \begin{align*} 
 |x-y|\ge|y|-|x| \ge \frac{|y|}2+\frac{R}2-\rho \ge\frac{\rho+|y|}2
 \geq c_\rho\cdot\frac{1+|y|}{2},
 \end{align*}
with $c_\rho := \min\{\rho,1\}$.
 Since $\varphi=1$ in $B(0,R)\supseteq\OO$, for every $x\in \OO$ we then get
\begin{equation}\label{cutoff-au-1}
\begin{split}
 & \bigg|\int_{\R^N}\frac{(f(x)-f(y))(\varphi(x)-\varphi(y))}{|x-y|^{N+2s}}\,\d y \bigg|
  \\[0.2cm]
  & \quad =\left|
  \int_{\{|y|\geq R\}}\frac{(f(x)-f(y))(1-\varphi(y))}{|x-y|^{N+2s}}\,\d y\right|
   \leq
 \int_{\{|y|\geq R\}}\frac{|f(x)-f(y)|}{|x-y|^{N+2s}}\,\d y \\[0.2cm] 
 & \quad \le|f(x)|
 \int_{\{|y|\geq R\}}\frac{\d y}{|x-y|^{N+2s}}+
\int_{\{|y|\geq R\}}\frac{|f(y)|}{|x-y|^{N+2s}}\,\d y \\[0.2cm]
 &\quad\le
 |f(x)|
 \int_{\{|z|\geq R/2\}}\frac{\d z}{|z|^{N+2s}}+(2/c_\rho)^{N+2s}
 \int_{\{|y|\geq R\}}\frac{|f(y)|}{( {1}+|y|)^{N+2s}}\,\d y\\[0.2cm]
 &\quad \le C\left(\frac{|f(x)|}{R^{2s}}+\kappa(R)\right),
\end{split}
\end{equation}
for some $C = C(N,s,\rho)>0$. Similarly, for every $x\in \OO$ we have
 \begin{equation}\label{cutoff-au-2}
 \begin{split}&
 \big|(-\Delta)^s\varphi(x)\big|=c_{N,s}\,\left|
 \int_{\{|y|\geq R\}}\frac{1-\varphi(y)}{|x-y|^{N+2s}}\,\d y\right|
 \\[0.1cm]
 & \qquad \le c_{N,s}\,
 \int_{\{|y|\geq R\}}\frac{\d y}{|x-y|^{N+2s}}
 \leq
 c_{N,s}\,      
 \int_{\{|z|\geq R/2\}}\frac{\d z}{|z|^{N+2s}} \le
 \frac{C}{R^{2s}},
 \end{split}
\end{equation} 
 up to renaming $C>0$.
 Since, obviously, 
 $\Delta f^\star=\Delta f$ on $\OO$ (remind that $\varphi\equiv 1$ on 
 $B(0,R)\subseteq\OO$), if $x\in \OO$ we obtain
 (see, e.g.,~\cite[formula~(2.11)]{BarPerSorVal})
 \begin{equation}\label{cutoff-au-3}
 \begin{split}
 &
  \big|{\mathcal{L}}f^\star(x)- {\mathcal{L}}f(x)\big| \\[0.1cm]
  & \qquad = \Big|(-\Delta)^sf^\star(x)- (-\Delta)^s f(x)\Big| \\[0.1cm]
  & \qquad = \bigg|\varphi(x)(-\Delta)^sf(x)+ f(x)(-\Delta)^s\varphi(x) \\[0.1cm]
 & \qquad\qquad\quad -c_{N,s}\int_{\R^N}\frac{(f(x)-f(y))(\varphi(x)-\varphi(y))}{|x-y|^{N+2s}}\,\d y
 - (-\Delta)^s f(x) \bigg|\\[0.1cm]
 & \qquad=\left|f(x)(-\Delta)^s\varphi(x)
	-c_{N,s}\int_{\R^N}\frac{(f(x)-f(y))(\varphi(x)-\varphi(y))}{|x-y|^{N+2s}}\,\d y\right| \\[0.1cm]
 & \qquad\le |f(x)|\,
 \big|(-\Delta)^s\varphi(x)\big|+c_{N,s}
  \left|\int_{\R^N}\frac{(f(x)-f(y))(\varphi(x)-\varphi(y))}{|x-y|^{N+2s}}\,\d y\right|.
 \end{split}
 \end{equation}
 The desired result thus follows by inserting~\eqref{cutoff-au-1}
 and~\eqref{cutoff-au-2}
 into~\eqref{cutoff-au-3}.
 \end{proof}
 
 We are finally ready to prove Theorem~\ref{BDTH}.
 
 \begin{proof}[Proof of Theorem~\ref{BDTH}]
The gist is that the function~$\gamma$ belongs to the space~$\mathcal{C}_s(\RN)$
defined in~\eqref{eq.defSpaceLLs} (because~$\gamma(x_1)$
grows like~$x_1^{\alpha_J}$ with $\alpha_J < 2s$ as~$ x_1 \to+\infty$,
the highest exponent coming from~\eqref{HIGHE}) and therefore
the corresponding quantity~$\kappa(R)$ introduced in Lemma~\ref{lem.cutOffpergamma}
goes to zero as~$R\to+\infty$. As a general philosophy,
the main difficulty with all of the truncations in the fractional setting is to control the
errors developed by nonlocality:
in this argument these errors are accounted for by the quantity~$\kappa(R)$, which becomes negligible for~$R$ large.
In this sense, taking care of the fact that the function~$\gamma$ grows slower than~$x_1^{
2s}$ once we leave the ``boundary layer''~$ \{x_1 < \ell\}$
pays off now, since it allows us to have full control of the ``contributions coming from infinity''. 

The technical details of the proof go as follows.
 Up to a rigid motion, we can assume that $p=0$
 and that
 \begin{equation}\label{8yiuqwgdugfew36rrthj1}
 \Omega\subset\{x\in\RN:\,x_1>0\}.
 \end{equation}
 Moreover, if $\ell > 0$ is as in Corollary~\ref{Y:3}, 
 we define $\Omega_*:=\Omega\cap\{x_1<\ell\}$.
 Notice that, on account of~\eqref{8yiuqwgdugfew36rrthj1}, we have
 \begin{equation}\label{8yiuqwgdugfew36rrthj}
 \Omega_*\subset\{x\in\RN:\,x_1\in(0,\ell)\}.
 \end{equation}
 We now choose 
 $\rho > 0$ in such a way that $\Omega\subseteq B(0,\rho)$, and we
 let $R > 0$ be such that $R > 4\rho$. Moreover,
 if $\gamma$ is as in Corollary~\ref{Y:3}, we define
 $$\Theta(x) := \gamma(x_1)\cdot\varphi(x) \qquad
 {\mbox{for any }} x\in\RN,$$
 where $\varphi = \varphi_R\in C_0^\infty(\RN,\R)$ 
 satisfies (i)-(ii)-(iii) in the statement of 
 Lemma~\ref{lem.cutOffpergamma}.
 Taking into account that 
 $\gamma\in \CCs(\R)\cap C^2((0,\ell),\R)$, $\varphi\in C_0^\infty(\RN,\R)$
 and $\varphi\equiv 1$ on the ball $B(0,R)\supseteq\Omega_*$,
 it is readily seen that 
 $$\Theta\in \CCs(\RN)\cap C^2(\Omega_*,\R).$$
 Moreover, since $\gamma\in H^1_\loc(\R)$ and $\mathrm{supp}(\varphi)\subseteq
 B(0,2R)$, we also have
 $$\Theta\in H^1(\RN).$$
 Finally, by combining~\eqref{EQUA:L-1},
 \eqref{EQUA:LR-1} and Lemma~\ref{lem.cutOffpergamma}, we obtain
 \begin{align*}
 {\mathcal{L}}\Theta(x)
   \geq{\mathcal{L}}\gamma(x_1)-
  C\left( \frac{|\gamma(x_1)|}{R^{2s}}+\kappa(R)\right)
   \geq 1 - C\left(\frac{\ell}{c}\cdot\frac{1}{R^{2s}}+\kappa(R)\right)
 \end{align*}
 for every $x\in\Omega_*\subseteq\{x\in\RN:\,x_1\in(0,\ell)\}$.
 In view of this last computation, by enlarging $R > 0$ if necessary, we get
 \begin{equation} \label{eq.LThetalarge}
  {\mathcal{L}}\Theta(x) \geq \frac{1}{2}\qquad\text{for
  every $x\in \Omega_*$}.
 \end{equation}
 We then turn to use $\Theta\in H^1(\RN)$ as a barrier to prove~\eqref{Y:1}. To this end, we consider
 the function~$v:\RN\to\R$ defined as
 $$ v(x):= u(x)- 4\bar C\,\Theta(x),$$
 where $\bar C$ is as in~\eqref{eq.PbGeneralWeak}.
 We observe that, by~\eqref{eq.LThetalarge}, we have
 \begin{equation}\label{1LAPbarr}
  \begin{split}
   {\mathcal{L}}v =
{\mathcal{L}}u-
   4\,\bar C\,{\mathcal{L}}\Theta \le \bar C-\frac{4\bar C}{2}\le0
 \end{split}
 \end{equation}
 in $\Omega_*$. Now we claim that, for a.e.\,$x\in{\RN}\setminus\Omega_*$, we have
 \begin{equation}\label{1234}
  v(x)\le0.
 \end{equation}
 To check this, we observe that
 $$ {\RN}\setminus\Omega_*\subseteq
 \big({\RN}\setminus\Omega\big)\cup \{x\in\Omega:\,x_1\ge\ell\}.$$
 Hence we distinguish two cases. 
 \begin{itemize}
  \item $x\in \RN\setminus\Omega$. In this case, since $u\equiv 0$ a.e.\,in $\RN\setminus\Omega$, we have
  $$v(x)=- 4\,\bar C\,\Theta(x)
  = -4\,\bar C\,(\gamma(x_1)\cdot\varphi(x));$$ 
  from this, since
  $\gamma\geq 0$ on $\R$ (by Corollary~\ref{Y:3}) and $0\leq\varphi\leq 1$
  (see (ii) in Lemma~\ref{lem.cutOffpergamma}), we derive that
  $v\leq 0$ a.e.\,in $\RN\setminus\Omega$. \vspace*{0.07cm}
  
  \item $x\in \{x\in\Omega:\,x_1\geq\ell\}$. In this case,
  using~\eqref{ZwdZsdZEQUA:L-1} and the fact that
  $\varphi\equiv 1$ on $B(0,R)\supseteq\Omega$, we can write
  \begin{align*}
   v(x) & = u(x) - 4\,\bar C\,\gamma(x_1) 
   \le u(x)- 4\,\bar C\le \bar C-4\,\bar C\le0,
  \end{align*}
  and this concludes the proof of~\eqref{1234}.
 \end{itemize}
 {F}rom~\eqref{1LAPbarr},~\eqref{1234} and Theorem~\ref{thm.WMPweak}
 (notice that $v\in H^1(\RN)$, since the same is true of
 both $u$ and $\Theta$), we conclude that 
 $v(x)\le0$ for a.e.\,$x\in{\RN}$.
 Hence, we obtain
 (see~\eqref{EQUA:LR-1})
 $$ u(x)\le 4\,\bar C\,\Theta(x)
 = 4\,\bar C\,\gamma(x_1) \le\frac{4\,\bar C\,x_1}c
 \leq \frac{4}{c}\cdot\bar C\,|x|
 $$
 for a.e.\,$x\in\Omega\cap B(0,\ell)
 \subseteq\{x\in\RN:\,x_1\in(0,\ell)\}$, and this establishes~\eqref{Y:1}.
\end{proof}
{F}rom Theorem~\ref{BDTH}, we immediately obtain the following result.
   \begin{corollary}
   Let $\Omega$ be open and strictly convex, and let~$f\in L^\infty(\Omega)$.
   Let $u_f\in\spX$ be the \emph{(}unique\emph{)} weak solution of
     problem~\eqref{eq.mainPB}.
     
     Then, there exists~$\ell>0$ such that, for every~$p\in\partial\Omega$, we have that
	\begin{equation} \label{Y:10}
    |{u}_f(x)|\le C\,
    \big(\|{u}_f\|_{L^\infty({\RN})}+\|f\|_{L^\infty(\Omega)}\big)
    \, |x-p|,\qquad {\mbox{for a.e. }} x\in B(p,\ell).
    \end{equation} 
   \end{corollary}
   \begin{proof}
 Formula~\eqref{Y:10}
 follows from Theorem~\ref{BOUNDED} and~\eqref{Y:1}, 
 applied to both ${u}_f$ and $-{u}_f$, choosing
 $$\bar C:=\|{u}_f\|_{L^\infty({\RN})}+\|f\|_{L^\infty(\Omega)}.$$
 This ends the proof.
 \end{proof}
 We point out that the term
	$$\|{u}_f\|_{L^\infty({\RN})}$$
	in~\eqref{Y:10} can be actually reabsorbed into $\|f\|_{L^\infty(\Omega)}$,
	as it follows from Theorem~\ref{BOUNDED}.
\medskip

As a byproduct of Theorem~\ref{BDTH}, we also establish Theorem~\ref{thm.RegulBoundaryII}:
 \begin{proof}[Proof of Theorem~\ref{thm.RegulBoundaryII}]
 	 First of all, since $f\in L^\infty(\Omega)$, we know from
 	 Theorem~\ref{thm.existenceLax} that there exists a (unique)
 	 weak solution $u_f\in \spX$ of problem~\eqref{eq.mainPB}. Moreover,
 	 by combining Theorems~\ref{BOUNDED} and~\ref{BDTH},
 	 we infer the existence of a suitable constant $\mathbf{c} > 0$,
 	 independent of $u_f$, such that
 	 \begin{itemize}
 	  \item[(a)] $\|{u}_f\|_{L^\infty(\RN)}\leq \mathbf{c}\,\|f\|_{L^\infty(\Omega)}$; 
 	  \item[(b)] there exists $\ell > 0$ such that, for every $p\in\de\Omega$,
	 $$|{u}_f(x)|\leq \mathbf{c}\,\|f\|_{L^\infty(\Omega)}\cdot
	 |x-p|\qquad\text{for a.e.\,$x\in\Omega\cap B(p,\ell)$}.$$    
 	 \end{itemize}
 	 Now, since $f\in C^k(\Omega,\R)$ (and $k\geq \frac{N}2+3$), we derive
 	 from Corollary~\ref{cor.regulSmoot} that there exists a unique function 
 	 $\widehat{u}\in C^k(\Omega,\R)$ such that
 	 \begin{equation} \label{eq.uhatufae}
 	  \widehat{u} \equiv u_f \qquad\text{a.e.\,on $\Omega$},
 	 \end{equation}
 	 where $k = k_{m,N}$ is as in~\eqref{eq.defkmN}. In particular,
 	 $k\geq 2$. Setting
 	 $$\mathfrak{u}_f:\RN\to\R, \qquad
 	 \mathfrak{u}_f(x) := 
 	 \begin{cases}
     \widehat{u}(x), & \text{for $x\in {\Omega}$}, \\
     0, & \text{for $x\notin {\Omega}$},
     \end{cases}
     $$
     we claim that $\mathfrak{u}_f$ is a classical solution
     of~\eqref{eq.mainPB}, further satisfying (i)-(ii)-(iii).
     Indeed, using~\eqref{eq.uhatufae} 
     and the fact that $u_f\equiv 0$ a.e.\,in $\RN\setminus\Omega$,
     we have
     $$\mathfrak{u}_f \equiv u_f \qquad\text{a.e.\,in $\RN$}.$$
     As a consequence, $\mathfrak{u}_f\in H^1(\RN)$
     (hence, (i) is fulfilled) and, since $u_f$
     satisfies (a)-(b), we im\-me\-dia\-te\-ly derive that
     $\mathfrak{u}_f$ satisfies (ii)-(iii) (with the same constants
     $\mathbf{c},\,\ell > 0$). In particular,
     from (iii) we deduce that
     $$\lim_{x\to p}\mathfrak{u}_f(x) = 0\qquad\text{for all $p\in\de\Omega$},$$
     thus, $\mathfrak{u}_f$ being bounded,
     we get $\mathfrak{u}_f\in \CCs(\RN)$.
     Finally, since $u_f$ is a weak solution of~\eqref{eq.mainPB}
     and $\mathfrak{u}_f\in \CCs(\RN)\cap C^2(\Omega,\R)$
     (actually, $\mathfrak{u}_f\in C^k(\Omega,\R)$), from (i) and
     Remark~\ref{rem.classical}
     we conclude that $\mathfrak{u}_f$ is a classical solution of~\eqref{eq.mainPB}. 
     The uniqueness of $\mathfrak{u}_f$
     then follows from Corollary~\ref{cor.uniqueClassical}, and the proof of
     Theorem~\ref{thm.RegulBoundaryII} is thereby complete.
 	\end{proof}
 	
 	 \appendix
    \section{Failure of the maximum principle} \label{sec.appendix}
       \subsection{The case of $\LL' := \Delta+(-\Delta)^s$}
    The following examples show that
    the weak maximum principle contained in~\eqref{thm.WMPLL} \emph{does not hold}
    for the operator
    \begin{equation}\label{d98bt549675 9}
    \LL' := \Delta+(-\Delta)^s.\end{equation}
     \begin{example} \label{exm.CES}
  Let $s\in(0,1/2)$ be arbitrarily fixed, and let 
  $$f:\R\longto\R, \qquad f(x) := \begin{cases}
  x^2-1, & \text{if $|x|\leq 1$}, \\
  0, & \text{if $|x| > 1$}.
  \end{cases}$$
  Clearly, $f\in C_b(\R)$ (as
  $|f|\leq 1$). Moreover, setting $\Omega_0 := (-1,1)$, we also
  have that~$f\in C^2(\Omega_0)$. 
  
  We claim that,
  for every fixed $x\in\R$, one has
  \begin{equation} \label{eq.toproveintegralf}
  y\mapsto\frac{f(x)-f(y)}{|x-y|^{1+2s}}\in L^1(\R).
  \end{equation}
  In order to prove~\eqref{eq.toproveintegralf}, we distinguish three cases:
  \begin{itemize}
   \item[(i)] $x\in(-1,1)$. In this case, reminding that
   $s\in(0,1/2)$, we have
  \begin{align*}
   & \int_{\R}\frac{|f(x)-f(y)|}{|x-y|^{1+2s}}\,\d y \\[0.1cm]
   & \qquad
   = |f(x)|\,\int_{-\infty}^{-1}
     \frac{\d y }{|x-y|^{1+2s}}
   + |f(x)|\,\int_1^\infty   \frac{\d y}{|x-y|^{1+2s}} +
   \int_{-1}^1\frac{|x^2-y^2|}{|x-y|^{1+2s}}\,\d y  \\[0.1cm]
   &  
    \qquad \leq\int_{-\infty}^{-1}
     \frac{\d y }{|x-y|^{1+2s}}
   + \int_1^\infty   \frac{\d y}{|x-y|^{1+2s}}
   + 2\int_{-1}^1\frac{\d y}{|x-y|^{2s}} 
   < \infty.
  \end{align*}
  \item[(ii)] $x = \pm 1$. We perform the computations when~$x=1$,
   being the case $x = -1$ completely analogous.
  In this case, we have
  \begin{align*}
   \int_{\R}\frac{|f(x)-f(y)|}{|1-y|^{1+2s}}\,\d y
   & = \int_{-1}^1\frac{1-y^2}{|1-y|^{1+2s}}\,\d y 
   \leq 2\int_{-1}^1\frac{1}{|1-y|^{2s}} < \infty .
  \end{align*}
  \item[(iii)] $x\notin[-1,1]$. In this case, since $|x-y|\geq |x|-1$ if
  $y\in(-1,1)$, we have
  \begin{align*}
   \int_{\R}\frac{|f(x)-f(y)|}{|x-y|^{1+2s}}\,\d y
   & = \int_{-1}^1\frac{1-y^2}{|x-y|^{1+2s}}\,\d y \leq
   \frac{1}{(|x|-1)^{1+2s}}\int_{-1}^1(1-y^2)\,\d y<\infty.
  \end{align*}
  \end{itemize}
  Summing up, the claimed~\eqref{eq.toproveintegralf} is completely established. \vspace*{0.02cm}
  
  Now, we observe that,
for any $x\in(-1,1)$,
  \begin{align*}
  & \frac{|(-\Delta)^sf(x)|}{c_{1,s}}
   \leq \int_{\R}\frac{|f(x)-f(y)|}{|x-y|^{1+2s}}\,\d y \\[0.1cm]
  & \qquad
  \leq (1-x^2)\int_{-\infty}^{-1}  \frac{\d y}{|x-y|^{1+2s}} +
   (1-x^2)
  \int_1^\infty
   \frac{\d y}{|x-y|^{1+2s}} + \int_{-1}^1\frac{|x^2-y^2|}{|x-y|^{1+2s}}\,\d y \\[0.1cm]
   & \qquad
  \leq (1-x^2)\int_{-\infty}^{-1}  \frac{\d y}{|x-y|^{1+2s}} +
   (1-x^2)
  \int_1^\infty
   \frac{\d y}{|x-y|^{1+2s}} + 2\int_{-1}^1\frac{\d y}{|x-y|^{2s}} \\[0.1cm]
  & \qquad\leq 2^{2-2s}\cdot\frac{1-s}{s(1-2s)}
 .
  \end{align*}
  As a consequence, if $\e\in(0,1)$ and if
  $f_\e(x) := f(x/\e)$, we have
  $$\LL'f_\e(x) = \frac{2}{\e^2} + \frac{1}{\e^{2s}}\big((-\Delta)^s f\big)(x/\e)
  \geq \frac{2}{\e^2}\bigg(1-\e^{2-2s}\cdot\frac{2^{1-2s}\, c_{1,s}\,(1-s)}{s(1-2s)}\bigg),$$
  for every $x\in\R$ with $|x|<\e$. If we choose $\e_0$ so small that
  $$1-\e_0^{2-2s}\cdot\frac{2^{1-2s}\,c_{1,s}\,(1-s)}{s(1-2s)} > 0,$$
  we thus see that $f_{\e_0}$ enjoys the following properties:
  \begin{itemize}
   \item[(a)] $f_{\e_0}\in C^2(\Omega_{\e_0})\cap C_b(\R)$, where 
   $\Omega_{\e_0} := (-\e_0,\e_0)$;
   \item[(b)] the $s$-Laplacian of $f_{\e_0}$ is pointwise defined on the whole of $\R$;
   \item[(c)] $f_{\e_0} \equiv 0$ on $\R\setminus\Omega_{\e_0}$ and
   $\LL'f_{\e_0} > 0$ on $\Omega_{\e_0}$.
  \end{itemize}
  Since, obviously, $f_{\e_0} < 0$ on $\Omega_{\e_0}$, we conclude that
  a weak maximum principle as in~\eqref{thm.WMPLL} does not hold
  for $\LL' = \Delta+(-\Delta)^s$.
 \end{example}
 
 \begin{example} \label{exm.generals}
  By dropping
  the assumption~$u \equiv 0$ in~$\RN\setminus\Omega$, it is possible
  to show that $\LL'$ in~\eqref{d98bt549675 9}
  violates the weak maximum principle
  in~\eqref{thm.WMPLL} \emph{for every $s\in(0,1)$}.
  Indeed, let $s\in(0,1)$ and let
  $$f:\RN\longto\R, \qquad f(x) := |x|^2-1.$$
  Moreover, let $\varphi\in C_0^\infty(\RN,[0,1])$ be a cut-off function
  such that 
  \begin{equation}\label{049vghjglfwveii}
  \text{$\varphi \equiv 1$ on $\Omega := B(0,1)$ and
  $\varphi \equiv 0$ on $\RN\setminus B(0,2)$}.\end{equation}
  We then set $u := f\varphi$. Obviously,
  $u\in C_0^\infty(\RN)\subseteq C^2(\Omega)\cap C_b(\RN)$. Moreover, 
  by taking into account the properties of~$\varphi$ in~\eqref{049vghjglfwveii}, we see that
  $$\text{$\Delta u = \Delta f = {2N}$ in $\Omega$ \qquad and \qquad
  $(-\Delta)^s u\in L^\infty(\RN)$.}$$
  As a consequence, if $\e\in(0,1)$ and $u_\e := u(x/\e)$, we have
  $$\LL'u_\e(x) = \frac{{2N}}
  {\e^2} + \frac{1}{\e^{2s}}\big((-\Delta)^s u\big)(x/
  {\e}) 
  \geq \frac{1}{\e^2}\Big({2N}-\e^{2-2s}\,\|(-\Delta)^s u\|_{L^\infty(\RN)}\Big),$$
  for all $x\in\Omega_\e := B(0,\e)$. We now argue as in Example~\ref{exm.CES}:
  if $\e_0$ is so small that
  $$ {2N}-\e_0^{2-2s}\,\|(-\Delta)^s u\|_{L^\infty(\RN)} > 0,$$
  then the function $u_{\e_0}$ enjoys the following properties:
  \begin{itemize}
   \item[(a)] $u_{\e_0}\in C_0^\infty(\RN)\subseteq C^2(\Omega_{\e_0})\cap C_b(\R)$;
   \item[(b)] the $s$-Laplacian of $u_{\e_0}$ is pointwise defined on the whole of $\R$
   (and it is globally bounded);
   \item[(c)] $u_{\e_0} \geq 0$ on $\R\setminus\Omega_{\e_0}$ and
   $\LL'u_{\e_0} > 0$ on $\Omega_{\e_0}$.
  \end{itemize}
  Since, obviously, $u_{\e_0} < 0$ on $\Omega_{\e_0}$, we conclude that
  a weak maximum principle as in~\eqref{thm.WMPLL} does not hold
  for $\LL' = \Delta+(-\Delta)^s$, for any~$s\in (0,1)$.
 \end{example}
 
 \subsection{The r\^ole of the `non-local boundary conditions'}
    Throughout the sequel, given any $R > 0$, we adopt the simplified notation
    $$B_R := B(0,R).$$
    Then, we claim that for every fixed $r > 1$ there exists
    $v\in C^2(B_1,\R)\cap \CCs(\RN)$ such that
    \begin{equation}\label{Dett-2}
	\begin{cases}
	{\mathcal{L}}v=0 & \text{in $B_1$},\\
	\displaystyle\inf_{B_r\setminus B_1}v > 0 ,\\
	\displaystyle\min_{B_1} v < 0.
	\end{cases}
	\end{equation}
	{F}rom this, since the continuity of $v$ implies that
	 $v\geq 0$ on $\de B_1$,
	we deduce that the weak maximum principle in~\eqref{thm.WMPLL}
	\emph{does not hold} if one requires $u\geq 0$ only on $\de\Omega$.
	\medskip
	
	\noindent To prove the existence of such a function $v$, 
	we let
	$\phi\in C^\infty_0(\RN,\R)$ be such that \medskip
	
	(a)\,\,$\phi\equiv -1$
	in $B_{r+3}\setminus B_{r+2}$ and $\phi\equiv 0$ outside $B_{r+4}\setminus B_{r+1}$;
	\medskip
	
	(b)\,\,$-1\leq\phi\leq 0$ on the whole of $\RN$. \medskip
	
	\noindent Since $\phi\in C_0^\infty(\RN,\R)\subseteq\mathcal{S}(\RN)$, we have that~$f := -\LL\phi$
	can be computed pointwise in $\RN$ and~$f\in C^\infty(\RN)\cap L^\infty(\RN)$.
	
	As a consequence, from Theorem~\ref{thm.RegulBoundaryII} we know that there exists
	a unique classical solution $\mathfrak{u}_\phi\in C^2(B_1,\R)\cap \CCs(\RN)$
	of the problem
	\begin{equation} \label{eq.PBsolveduphi} 
	\begin{cases}
	 \LL u = f = -\LL\phi & \text{in $B_1$}, \\
	 u\equiv 0 & \text{on $\RN\setminus B_1$}.
	\end{cases}
	\end{equation}
	We then set $w := \mathfrak{u}_\phi+\phi$
	and we notice that, thanks to the regularity of $\mathfrak{u}_\phi$ and $\phi$,
	one has that~$w\in C^2(B_1,\R)\cap \CCs(\RN)$. Furthermore, from~\eqref{eq.PBsolveduphi}
	we obtain that
	\begin{equation}\label{Dett-3}	
	\begin{cases}
	\LL w=0 & \text{in $B_1$}, \\
	w\equiv \phi & \text{in $\RN\setminus B_1.$}
	\end{cases}	
	\end{equation}
	We now claim that
	\begin{equation}\label{mino9ks}
	m := \inf_{B_1} w < 0.
	\end{equation}
	For this, we argue by contradiction and we suppose that
	$$w(x)\ge0\qquad\text{for every $x\in B_1$}.$$
    In particular, since $w\in C(\RN,\R)$ and $w\equiv\phi\equiv 0$ on $\de B_1$, 
    we can find an \emph{interior} maximum point~$x_0\in B_1$ for $w$ such that 
    \begin{equation}\label{f9854yhgikfrh}
    w(x_0)\geq 0.\end{equation}
    Thus,
    $\Delta w(x_0)\leq 0$ and therefore, using the exterior condition in~\eqref{Dett-3},
	\begin{equation}\label{jd49837btheryehy07}\begin{split}
	& 0 =\frac{ {\mathcal{L}}w(x_0)}{c_{N,s}}= \frac{\big(-\Delta+(-\Delta)^s\big)w(x_0)}{c_{N,s}} \geq \frac{(-\Delta)^sw(x_0)}{c_{N,s}}=
	\int_{\R^N}
	\frac{w(x_0)-w(y)}{|x_0-y|^{N+2s}}\,\d y \\[0.1cm]
	& \qquad=\int_{B_1}\frac{w(x_0)-w(y)}{|x_0-y|^{N+2s}}\,\d y
	+
	\int_{\R^N\setminus B_1}
	\frac{w(x_0)-\phi(y)}{|x_0-y|^{N+2s}}\,\d y\\[0.1cm]
	& \qquad \ge \int_{\R^N\setminus B_1}
	\frac{w(x_0)-\phi(y)}{|x_0-y|^{N+2s}}\,\d y,
		\end{split}\end{equation}
	where we used that fact that~$w(x_0)\geq w(y)$ for all $y\in B_1$ in the last line.
	 Moreover, by assumption~(a), we know that~$\phi\equiv -1$
	on~$B_{r+3}\setminus B_{r+2}$}. Accordingly, using also~\eqref{f9854yhgikfrh} and assumption~(b)
	on~$\phi$, we conclude from~\eqref{jd49837btheryehy07} that
	
	\begin{align*}&
	0\geq \int_{B_{r+3}\setminus B_{r+2}}
	\frac{w(x_0)+1}{|x_0-y|^{N+2s}}\,\d y+
	\int_{(B_{r+2}\setminus B_1)\cup (\RN\setminus B_{r+3})}
	\frac{w(x_0)-\phi(y)}{|x_0-y|^{N+2s}}\,\d y\\
	&\qquad \ge\int_{B_{r+3}\setminus B_{r+2}}
	\frac{w(x_0)+1}{|x_0-y|^{N+2s}}\,\d y > 0.
	\end{align*}
	This contradiction proves~\eqref{mino9ks}. 
	With~\eqref{mino9ks} at hand, we define
	$$ u(x):=w(x)-m,$$
	and we observe that, in view of the properties of $w$, one has \medskip
	
	(1)\,\,$u\in C^2(B_1,\R)\cap\CCs(\RN)$ (as the same is true of $w$); 
	
	(2)\,\,$\LL u = \LL w = 0$ pointwise on $B_1$; 
	
	(3)\,\,$u = -m > 0$ on $B_r\setminus B_1$
	(as $w\equiv\phi\equiv 0$ on $B_r\setminus B_1$); 
	
	(4)\,\,$\inf_{B_1} u = 0$. \medskip
	
    \noindent Thus, by making use of (1)--(4), 
    we easily conclude that the function
	$$v(x):=2u(x)+m,$$
	belongs to $C^2(B_1,\R)\cap\CCs(\RN)$ and satisfies~\eqref{Dett-2}.

\vfill
\end{document}